\documentclass[journal,12pt,onecolumn,compsoc]{IEEEtran}

\usepackage[english]{my-shortcuts}
\usepackage{srcltx,algorithm,algorithmic}

\begin{document}

\title{Sparse recovery with unknown variance:\\ a LASSO-type approach}

\author{St\'ephane Chr\'etien and S\'ebastien Darses}

\author{St\'ephane~Chr\'etien,
        S\'ebastien~Darses 
\IEEEcompsocitemizethanks{\IEEEcompsocthanksitem S. Chr\'etien is with the Laboratoire de Math\'ematiques, UMR 6623,
Universit\'e de Franche-Comt\'e, 16 route de Gray,
25030 Besancon, France \protect\\
E-mail: stephane.chretien@univ-fcomte.fr
\IEEEcompsocthanksitem S. Darses is with the LATP, UMR 6632, Universit\'e Aix-Marseille, Technop\^ole Ch\^{a}teau-Gombert,
39 rue Joliot Curie, 13453 Marseille Cedex 13, France\protect\\
E-mail: darses@cmi.univ-mrs.fr}
\thanks{Manuscript received ?, 2011; revised ?, 2012.}}

\IEEEcompsoctitleabstractindextext{
\begin{abstract}
We address the issue of estimating the regression vector $\beta$ in the generic $s$-sparse linear model $y = X\beta+z$, 
with $\beta\in\R^{p}$, $y\in\R^{n}$, $z\sim\mathcal  N(0,\sg^2 I)$ and $p> n$ when the variance $\sg^{2}$ is unknown. 
We study two LASSO-type methods that jointly estimate $\beta$ and the variance. These estimators are minimizers 
of the $\ell_1$ penalized least-squares functional, where the relaxation parameter is tuned according 
to two different strategies. In the first strategy, the relaxation parameter is of the order 
$\ch{\sigma} \sqrt{\log p}$, where $\ch{\sigma}^2$ is the empirical variance. 
In the second strategy, the relaxation parameter is 
chosen so as to enforce a trade-off between the fidelity and the penalty terms at optimality. For both estimators, our assumptions are similar to the ones proposed by Cand\`es and Plan in {\it Ann. Stat. (2009)}, for the case where $\sg^{2}$ is known. We prove that our estimators ensure exact recovery of the support and sign pattern of $\beta$ with high probability.
We present simulations results showing that the first estimator enjoys nearly the same performances in practice as the standard LASSO (known variance case) for a wide range of the signal to noise ratio. 
Our second estimator is shown to outperform both in terms of false detection, 
when the signal to noise ratio is low. 
\end{abstract}

}

\maketitle


\IEEEdisplaynotcompsoctitleabstractindextext

%
\IEEEpeerreviewmaketitle

\section{Introduction}

\subsection{Problem statement}
The well-known standard Gaussian linear model reads
\bea
y & = & X\beta+z,
\eea
where $X$ denotes a $n\times p$ design matrix, $\beta\in\R^{p}$ is an unknown parameter and the components of the error $z$ are assumed i.i.d. with normal distribution $\mathcal N(0,\sigma^2)$. The present paper aims at studying this model in the case where the number of covariates is
greater than the number of observations, $n< p$, and the regression vector $\beta$ and the variance $\sigma^2$ are both unknown.

The estimation of the parameters in this case is of course impossible without further assumptions on the regression 
vector $\beta$. One such
assumption is sparsity, i.e. only a few components of $\beta$ are different from zero, say $s$ components; $\beta$ 
is then said to be $s$-sparse. There has
been a great interest in the study of this problem recently. Recovering the support of $\beta$ has been extensively studied in the context of Compressed Sensing, a new paradigm for designing observation
matrices $X$. In this framework, it is now a standard fact that matrices $X$ can be
found (e.g. with high probability if drawn from sub-Gaussian i.i.d distributions) such that the number
of observations needed to recover $\beta$ exactly is proportional to $s \log\left(p/n \right)$. 

\subsection{Existing results in the known variance case}
When the variance is known and positive, two popular techniques to estimate the regression vector $\beta$ are 
the Least Absolute Shrinkage and Selection Operator (LASSO) \cite{Tibshirani:JRSSB96}, and the Dantzig selector \cite{CandesTao:AnnStat07}. 
We refer to \cite{Bickel:AnnStat09} for a recent simultaneous analysis of these two methods. 
The standard LASSO estimator $\ch{\beta}_\lambda$ of $\beta$ is defined as
\bea
\label{lass}
\ch{\beta}_{\lb} & \in & \down{b\in \mathbb R^p}{\rm argmin} \ \frac{1}{2} \|y-X b\|_2^2+ \lambda \|b\|_1,
\eea
where $\lambda>0$ is a regularization parameter controlling the sparsity of the estimated coefficients.

Sparse recovery cannot hold without some geometric assumptions on the dictionary (or 
the design matrix), as recalled in \cite{Koltchinskii:AnnStat09} pp. 4--5. The papers 
\cite{Wainwright:IEEEIT09} and \cite{ZhaoYu:JMLR06} introduced very pertinent assumptions 
for the study of variable selection problem using the LASSO in the finite sample (resp. asymptotic) contexts. 

One common assumption for the precise study of the statistical performance of these estimators is an incoherence 
property of the matrix $X$. This means that the coherence of $X$, i.e. the maximum scalar product of two (normalized) columns of $X$, is very 
small. Coherence based conditions appeared first in the context of Basis Pursuit for sparse approximation in \cite{Donoho:IEEEIT01}, \cite{Elad:IEEEIT02} 
and \cite{Donoho:PNAS03}. 
It then had a significant impact on Compressed Sensing; see \cite{Rudelson:IMRN05} and \cite{Candes:InvProb07}. 

The recent references  
\cite{Bickel:AnnStat09}, \cite{Bunea:EJStat07} and \cite{Kerkyacharian:Conf09} contain 
interesting assumptions on the coherence in our context of interest, i.e. high dimensional sparse regression. 
For instance, \cite{Bickel:AnnStat09} and \cite{Bunea:EJStat07} require a 
bound of the order $\sqrt{\log n/n}$ whereas \cite{Kerkyacharian:Conf09} requires a bound of the order $1/s$. 
The recent paper \cite{CandesPlan:AnnStat09} requires that the coherence of $X$ is less than ${\rm Cst}/\log p$. 
Under the additional assumptions that $\beta$ is sparse and assuming that the support and sign pattern are 
uniformly distributed, they prove that $\ch{\beta}$ has the same support and sign pattern as $\beta$ with probability 
$1-p^{-1} ((2\pi \log p)^{-1/2}+sp^{-1})-O(p^{-2\log 2})$.   

Notice, as commented on in e.g. \cite{CandesPlan:AnnStat09}, 
that various assumptions in the literature, such as the invertibility of the restricted 
covariance matrix \cite{Wainwright:IEEEIT09} indexed by the signal's true support 
and the Irrepresentable Condition in \cite{ZhaoYu:JMLR06} can be derived from their 
incoherence condition, although with possibly suboptimal orders in certain instances.

\subsection{Existing results in the unknown variance case}
The problem of estimating the variance in the sparse regression model has been addressed in only a few references 
until now. In \cite{Baraud:AnnStat09} the authors analyze in the unknown variance setting AIC, BIC and AMDL based estimators, as well as estimators using a more general complexity penalty. As well known among practitioners, the LASSO procedure, at the price of certain assumptions on $X$, 
avoids the enumeration of all subsets of covariates, an intractable task when the number of covariates is large. This last property motivates the theoretical analysis provided in the present paper. 

In a recent work \cite{bsv}, a joint 
estimation procedure for both the regression vector and the variance is proposed. The authors give a detailed study of the 
risk under quite general conditions. In \cite{Sun:Test10}, 
it is proven in particular that, for the variance estimator of \cite{bsv}, under a compatibility condition 
introduced in \cite{vdGeer:EJS09}, $\lb \|\beta\|_1/\sg =o(1)$ if and only if $\ch \sg/  \sg = (1+o_{\bP} (1))$, for $\lambda$ such that 
$\bP (\lambda >a \|X^t (Y -X \beta )/n\|_\infty/\sg)\rightarrow 1$ where $a>1$ is any constant. The problem of support and sign pattern 
recovery as well as the one of providing non-asymptotic results with explicit constants are not addressed. 

\subsection{Our contribution}
\label{contrib}

We study two different strategies in the present paper.

\subsubsection{Strategy (A): Plugging in the variance estimator}
Our work mainly aims at understanding when the results of \cite{CandesPlan:AnnStat09} extend to the case where 
$\sg^2$ is unknown. In the case where $\sigma^2$ is known, it is proven in \cite{CandesPlan:AnnStat09} that the right order of magnitude for $\lambda$ is $\sigma \sqrt{\log p}$.
We first study the very natural estimator consisting of replacing $\sigma$ by $\ch{\sigma}=\|y-X\ch{\beta}\|_2/\sqrt{n}$ 
in the expression of $\lambda$. As is standard in the study of the LASSO, the regression vector $\beta$'s coefficients have 
to be significantly larger than the noise level for exact recovery of the support and sign pattern.\\

The main differences between the known and the unknown variance cases are summarized in the following table.

$$
\begin{array}{c|c}
\text{Known variance} & \text{Unknown variance: Strategy (A)}\\
\hline
\\
\ch{\beta} \in \down{b\in \R^p}{\rm argmin} \ \frac{ \|y-X b\|_2^2}{2} + \lambda \|b\|_1 & 
\ch{\beta}_\lb \in \down{b\in \R^p}{\rm argmin} \ \ 
\frac{\|y-X b\|_2^2}{2} +\lb  \|b\|_1 \\
\\
\hline
\\
& {\rm  \textrm{ Tune } \lb \textrm{ to } \ch{\lb} \textrm{ s.t. }:\ } \ch{\lb}= C_{\rm var} \ch{\sg} \sqrt{\log p} \\
\lb={\rm cst}\ \sg \sqrt{\log p} & \\
& \quad {\rm with:\ }\ch{\sg}^2 = \frac{\|y-X \ch{\beta}_{\ch{\lb}}\|_2^2}{n} \\
\\
\hline
{\rm Convex\ problem} & {\rm Non\ convex \ problem}\\
\hline
\\
{\rm Oracle\ } \wdt\beta  & {\rm Oracle\ } (\wdt\beta, \wdt \lb) \\
\\
\hline
{\rm Conditions\ holding \ with } & \\
{\rm high\ probability } &  {\rm Similar \ conditions} \\
\end{array}
$$

\bigskip

Notice that, in this table, $\ch{\beta}$ is defined via $\ch{\lb}$ and $\ch{\lb}$ is defined 
via $\ch{\beta}$. In other words, $\ch{\beta}$ and $\ch{\lb}$ jointly satisfy a set of optimality 
conditions. From a numerical viewpoint, $\ch{\beta}$ and $\ch{\lb}$ can be computed iteratively 
using a fixed point-type algorithm; see Section \ref{fixedpt}.  

\subsubsection{Strategy (B): Enforcing a trade-off between fidelity and penalty}
\label{snr}
Another possible strategy can be used to overcome the problem of estimating the regression vector $\beta$ and 
the relaxation parameter $\lb$ when the variance $\sigma^2$ is unknown. This  
strategy consists of prescribing a 
trade-off between the fidelity term and the penalty term. 
More precisely, we will impose the constraint $\ch{\lambda} \|\ch{\beta}\|_1/\|y-X \ch{\beta}\|_2^2=C$. Enforcing 
such a trade-off between fidelity and penalty results in a more complex problem from both the statistical 
and the computationaly viewpoints. However, since $\ch{\lambda} \|\ch{\beta}\|_1$ and $\|y-X \ch{\beta}\|_2^2$ 
are, at least approximately, homogeneous functions of $\sigma^2$, using such a criterion allows to bypass the estimation
of the variance in a first stage. The variance itself could be estimated in a second stage, using the formula 
$\ch{\sg}^2 = \frac{\|y-X \ch{\beta}_{\ch{\lb}}\|_2^2}{n}$.

$$
\begin{array}{c|c}
\text{Known variance} & \text{Unknown variance: Strategy (B)}\\
\hline
\\
\ch{\beta} \in \down{b\in \R^p}{\rm argmin} \ \frac{ \|y-X b\|_2^2}{2} + \lambda \|b\|_1 & 
\ch{\beta}_\lb\in \down{b\in \R^p}{\rm argmin} \ \ 
\frac{\|y-X b\|_2^2}{2} + \lb \ \|b\|_1 \\
\\
\hline
\\
& {\rm  \textrm{ Tune } \lb \textrm{ to } \ch{\lb} \textrm{ s.t. }:\ }  \\
\lb={\rm cst}\ \sg \sqrt{\log p} & \\
& \ch{\lb} \|\ch{\beta}_{\ch{\lb}}\|_1 = C \ \|y-X \ch{\beta}_{\ch{\lb}}\|_2^2 \\
\\
\hline
{\rm Convex\ problem} & {\rm Non\ convex \ problem}\\
\hline
\\
{\rm Oracle\ } \wdt\beta  & {\rm Oracle\ } (\wdt\beta, \wdt \lb) \\
\\
\hline
\\
{\rm Conditions\ holding \ with } & {\rm Similar \ conditions } \\ 
{\rm high \ probability } &   {\rm + \ Upper \ bound \ on \ } \|\beta\|_1 \\
\end{array}
$$

\bigskip

\subsubsection{Results}

Our main results are Theorem \ref{main2}, for Strategy (A), and Theorem \ref{main3}, for Strategy (B).
Both results can be described as follows. Given an arbitrary $\alpha>0$, 
we prove that, for regression vectors $\beta$ satisfying certain natural constraints, 
standard assumptions on the number of observations $n$ and the 
sparsity $s$ imply that our modified LASSO procedures fail to identify the support and the signs of 
$\beta$ with probability at most of the order $p^{-\alpha}$. These results are non-asymptotic and all our constants are explicit. 

The coherence assumption on the design matrix made in this paper is readily checkable. Many other currently used 
assumptions in the literature are based on concentration properties of the extreme singular values of 
all or most extracted submatrices of $X$ with bounded number of columns. Yet, some other are based on 
the concentration of the singular values of the covariance matrix with respect to the covariate's underlying 
distribution. All such criteria are difficult or impossible to check in practice as opposed to the coherence property.

We neither make any uncheckable assumption on the variance $\sg^2$. 
The only unverifiable assumptions used in the present work are on the magnitude of the nonzero regression coefficients.  
As in \cite{CandesPlan:AnnStat09}, the set of regressors $\beta$ which are correctly estimated is constrained 
by imposing that the magnitude of all nonzero components of $\beta$ should be greater than the noise level. 
Moreover, for Strategy (B), our analysis requires the additional assumption that the components 
of $\beta$ should not be too large either, the upper bound being in particular a function of $C$. This result 
suggests that Strategy (B) is pertinent in low SNR situations only. Simulation experiments at the end
of this paper confirm the usefulness of Strategy (B) in the low SNR setting.

\subsection{Plan of the paper}

The LASSO estimator, the main results Theorem \ref{main2} and Theorem \ref{main3}, together with the assumptions used throughout the paper are presented 
in Section \ref{lasso}. The proof of Theorem \ref{main2} is given in Section \ref{Proofmain2} and the proof of 
Theorem \ref{main3} in Section \ref{Proofmain3}. The proofs of certain technical intermediate results are
gathered in the Appendix. 

\subsection{Notations} 
\subsubsection{Generalities}
When $E\subset \left\{1,\ldots,N\right\}$, we denote by $|E|$ the cardinal of $E$.
For $I\subset \left\{1,\ldots,p\right\}$ and $x$ a vector in $\R^{p}$, we set $x_{I}=(x_{i})_{i\in I}\in\R^{|I|}$.
The usual scalar product is denoted by $\la\cdot,\cdot\ra$. The notations for the norms on vectors and matrices 
are also standard: for any vector
$x=(x_{i})\in\R^{N}$,
\bean
\|x\|^{2}_{2} = \sum_{1\le i \le N}x_{i}^{2}\ ; \quad
\|x\|_{1} = \sum_{1\le i \le N}|x_{i}|\ ; \quad
\|x\|_{\infty} = \sup_{1\le i \le N}|x_{i}|.
\eean

For any matrix $A\in \R^{d_1\times d_2}$, we denote by $A^{t}$ its transpose. The set of symmetric real matrices in $\R^{n\times n}$ is denoted by 
$\mathbb S_n$. We denote by $\|A\|$ the operator norm of $A$. 
The maximum (resp. minimum) singular value of $A$ is denoted by $\sigma_{\max}(A)$ (resp. $\sigma_{\min}(A)$). 
Recall that $\sigma_{\max}(A)=\|A\|$ and, if $A$ is invertible, $\sigma_{\min}(A)^{-1}=\|A^{-1}\|$.
We use the Loewner ordering on symmetric real matrices: if $A\in \mathbb S_n$, 
$0\preceq A$ is equivalent to saying that $A$ is positive semi-definite, and $A\preceq B$ stands for $0\preceq B-A$.

The notations $\mathcal  N(\mu,\sg^{2})$ (resp. $\chi^2(\nu)$ and $\mathcal  B(\nu)$) stands for the normal distribution 
on the real line with mean $\mu$ and variance $\sg^{2}$ (resp. the Chi-square distribution with $\nu$ degrees of freedom and the Bernoulli distribution with parameter $\nu$).

\subsubsection{Specific notations related to the design matrix $X$ and the estimators}
For $I\subset \left\{1,\ldots,p\right\}$, and a matrix $X$,
we denote by $X_{I}$ the submatrix whose columns are indexed by $I$. We denote
the range of $X_I$ by $V_I$ and the orthogonal projection onto $V_I$ by $\gP_{V_I}$.

The coherence $\mu(X)$ of a matrix $X$ whose columns are unit-norm is defined by
\bea
\mu(X) & = & \max_{1\le i\neq j \le p} |\la X_{i},X_{j} \ra|.
\eea
As in \cite{Tropp:CRAS08}, we consider the 'hollow-gram' matrix $H$ and the selector matrix $R={\rm diag}(\delta)$:
\bea
H &=& X^tX-\Id \label{holo}\\
R &=& {\rm diag}(\delta), \label{selector}
\eea
where $\delta$ is a vector of length $p$ whose components are i.i.d. random variables following the Bernoulli distribution $ \mathcal B(s/p)$. In a similar fashion, we define $R_s={\rm diag}(\delta^{(s)})$
where $\delta^{(s)}$ is a random vector of length $p$, uniformly distributed on the set of all vectors with exactly $s$ 
components equal to 1 and $p-s$ components equal to 0.

The support of $\ch{\beta}$ is always denoted by $\ch{T}$. 


\section{The modified LASSO estimators}
\label{lasso}

In this section, we present the main results on the estimators given by Strategy (A) and Strategy (B),
and we discuss the underlying assumptions. Practical computability of these estimators will 
be studied in Section \ref{Simuls}. In particular "tuning $\lb$ to $\ch{\lb}$" is achieved by finding 
a zero of a function of $\lb$ numerically. We will show in Section \ref{Simuls} that these 
zero finding problems are computationally very easy to solve. 

For any arbitrary value of $\alpha>0$, Theorem \ref{main2} (resp. Theorem \ref{main3}), 
proposes a set of conditions under which exact recovery of the support and sign pattern of $\beta$ 
holds with probability at least $1-O(p^{-\alpha})$ for Strategy (A) (resp. for Strategy (B)). 

As will be shortly seen, the magnitude of the nonzero 
coefficients of $\beta$ has to satisfy certain constraints: 
as in \cite{CandesPlan:AnnStat09}, one will require for both Strategies that 
the nonzero components of $\beta$ are not too small (in fact, slightly above the noise level). In the case
of Strategy (B), we will moreover require that the nonzero components of $\beta$ are not too large. 
Although this upper bound assumption may seem to argue in disfavor of Strategy (B), computational
experiments will later show that this Strategy has much nicer empirical performance when the signal 
to noise ratio is small. The same computational experiments will also demonstrate that Strategy (A) 
performs almost as well as a standard LASSO which would know the variance.

\subsection{Definition of the estimators}

To define our estimators, we first need to work with matrices ensuring that the map 
$\lb\mapsto \ch{\beta}_\lb$, where $\ch{\beta}_\lb$ is given by (\ref{lass}), 
is well defined and enjoys special properties, 
such as continuity. \\

\begin{defi}
The matrix $X$ is said to satisfy the Generic Condition if 
\bea
\label{gc}
\left|\la x_j,X_I(X_I^tX_I)^{-1} \delta_I \ra \right|<1, \ \forall \delta \in \{-1,1\}^p, \ 
\forall I \subset \{1,\ldots,p\} \textrm{ s.t. } X_I \textrm{ non singular and } 
\forall j\not\in I. 
\eea
\end{defi}

As from now, we always work under the Generic Condition. We will use the following result about uniqueness of the LASSO estimator. \\

\begin{prop}[\cite{Dossal:HAL11}]
\label{propun}
Assume that $X$ satisfies the Generic Condition. 
Then, for all $y\in \R^n$, and for all $\lb \in \R_+$, Problem (\ref{lass}) has a unique solution 
$\ch{\beta}_\lb$ and its support $\ch{T}_\lb$ is such that $X_{\ch{T}_\lb}$ is non singular. 
\end{prop}
\

The following property is proven in Appendix \ref{lem-continuity}:\\
\begin{lemm}
\label{cont}
Let the Generic Condition hold. Then, almost surely, the map 
\bef{}
(0,+\infty) & \longrightarrow & \R^{p}\\
\nonumber
\lambda & \longmapsto & \ch{\beta}_{\lambda}
\eef
is bounded and continuous. Moreover, its $\ell_{1}$-norm is non-increasing.
\end{lemm}

\subsubsection{Strategy A}
\label{defStratA}
The estimator of strategy A is defined as $\ch \beta:=\ch \beta_{\ch \lb}$ where $\ch\lb$ verifies the implicit equation
\bea \label{eqA}
\ch\lb^2 & = & C_{{\rm var}}  \frac{\left\|y-X \ch{\beta}_{\ch \lb}\right\|_2^2}{n} \log p.
\eea
The estimators $(\ch \beta, \ch \lb)$ being implicitly defined, it is not clear, at that point, that they exist.\\
  
We will see in the sequel that a suitable choice of $C_{{\rm var}}$ will ensure the existence of the estimators (under the above mentioned assumptions on $X$).

The uniqueness follows by showing that the map $\Gamma_A: \R_+ \to \R_+$ given by 
\bean
\Gamma_A(\lb) & := & \frac{n}{\log p}\ \frac{ \lb^2}{\|y-X_{\ch{T}_\lb}\ch{\beta}_{\ch{T}_\lb}\|_2^2},
\eean
is increasing, which is proven in Appendix \ref{gammaA}.

Strategy A simply reduces to finding the value $\ch{\lb}_A$ such that 
$\Gamma_A(\ch{\lb}_A)=C_{{\rm var}}$. 
A precise range of interest for $C_{\rm var}$ will be given in Theorem \ref{main2} below. 
Moreover, using the existence and uniqueness result, 
one can use a fixed point scheme to find $\ch{\lb}$. This 
scheme is discussed in Section \ref{algoA}. \\

\begin{rema} \it
Recall that in the known variance case, it is often assumed that 
\bea
\label{lbchoice}
\lb^2 & = & C_{{\rm var}} \sg^2 \log p,
\eea 
for some positive constant $C_{{\rm var}}$; see e.g. in \cite{CandesPlan:AnnStat09}. 
In comparison, Strategy (A) enforces the choice (\ref{eqA}).
This is the empirical analog to (\ref{lbchoice}). However, as will appear later in the proof of Theorem \ref{main2}, 
instead of being an absolute constant, $C_{{\rm var}}$ will have to depend on $n$, $p$ and $\|X\|^2$ as
follows
\bean
C_{{\rm var}} & \asymp & \frac{n}{p} \|X\|^2.
\eean 
In the case of an i.i.d. Gaussian random design matrix, $\|X\|^2$ is of the order $p/n$ with high probability.
Thus $C_{{\rm var}}$ can be basically 
seen as a constant in the Gaussian setting. 
\end{rema}

\subsubsection{Strategy B}
\label{defStratB}
The estimator of strategy B is defined as $\ch \beta:=\ch \beta_{\ch \lb}$ where $\ch \lb$ verifies the implicit equation
\bea
\ch\lb  \|\ch{\beta}_{\lambda}\|_1& = & C \left\|y-X \ch{\beta}_{\ch \lb}\right\|_2^2.
\eea
Again, the estimators $(\ch \beta, \ch \lb)$ are implicitly defined and their existence has to be proven.\\
  
Compared to Strategy A, one specificity of Strategy B is that for 
any value of $C>0$, existence and uniqueness of the estimators is garanteed, 
with no other assumptions than the Generic Condition. 
Indeed, we show here (cf Lemma \ref{gammafunc} in the Appendix) that the map $\Gamma_B$ given by 
\bea
\label{gammaB}
\Gamma_B(\lb) & = & 
\frac{\lambda\|\ch{\beta}_{\lambda}\|_1}{\|y-X \ch{\beta}_{\lambda}\|_2^2}, \quad \lb>0,
\eea
is increasing, continuous and $\Gamma_B((0,+\infty))=(0,+\infty)$. Thus, there exists a unique value 
$\ch{\lb}_B>0$ such that $\Gamma_B(\ch{\lb}_B)=C$. 

Similarly as for Strategy A, a fixed point scheme will be discussed in Section \ref{Simuls}. \\

\subsection{Main results}

\subsubsection{Preliminary remarks}
\label{rem}
The main idea behind the analysis of LASSO-type methods is the following. First, the $\ell_1$ penalty 
promotes sparsity of the estimator $\ch{\beta}$. Since $\ch{\beta}$ is sparse, we may restrict the study 
to the subvector $\ch{\beta}_{\ch{T}}$ of $\ch \beta$, resp. the submatrix $X_{\ch{T}}$ of $X$, whose components, 
resp. columns, are indexed by $\ch{T}$. 

Taking this idea a little further, since $\ch{T}$ is supposed 
to estimate the true support $T$ of cardinality $s$, 
the first kind of result one may ask for is a proof that $X_T$ is far from singular for every possible $T$. Unfortunately, proving such a strong property with the right order in the upper bound on $s$, 
based on incoherence only, seems to be impossible. The idea proposed by Cand\`es and Plan in \cite{CandesPlan:AnnStat09} to overcome this problem is to assume that $T$ is random and then prove 
that non-singularity occurs with high probability, i.e. for most supports. 

Based on this model, the method first consists of proving that $X_T$ satisfies, for $0<r<1$,
\bea
\label{isomr0} 
1-r  \le \sigma_{\min}(X_T)\le \sigma_{\max}(X_T)  \le 1+r,
\eea  
with high probability. The proof of this property in \cite{CandesPlan:AnnStat09} is based on the Non-Commutative 
Kahane-Kintchine inequalities. In the present paper, we instead use a result of \cite{cd_inv} based on a recent version of the Non-Commutative 
Chernoff inequality proposed by Tropp \cite{Tropp:ArXiv10}, in order to obtain better estimates for the involved constants. 
The most intuitive conditions to prove (\ref{isomr0}) are:
\begin{itemize}
\item[(i)] $T$ is a random support with uniform distribution on index sets with cardinal $s$;
\item[(ii)]  $s$ is sufficiently small;
\item[(iii)]  $X$ is sufficiently incoherent.
\end{itemize}

The main part of the analysis consists of proving that the least-squares oracle estimator, 
which knows the support ahead of time, satisfies the optimality conditions of the LASSO estimator
with high probability. This will prove that the LASSO 
automatically detects the right support and sign pattern. The proofs of these results highly depend on the 
quasi-isometry condition (\ref{isomr0}) and similar properties obtained with the same techniques 
as for (\ref{isomr0}). We also need the sign pattern of $\beta$ to be uniformly distributed and jointly independent of the support of $T$. This assumption was
already invoked in \cite{CandesPlan:AnnStat09}.

\subsubsection{Assumptions and main results}

The first so-called Coherence condition deals with the 
minimum angle between the columns of $X$.
\begin{ass}{\rm (Range and Coherence condition for $X$)}
\label{ass1}
The matrix $X$ has unit $\ell_2$-norm columns, is full rank and its coherence verifies 
\bean
\mu(X) & \le & \frac{C_\mu}{\log p},
\eean
for some numerical constant $C_\mu>0$.
\end{ass}

\bigskip

\begin{ass}{\rm (Generic sparse model \cite{CandesPlan:AnnStat09})}
\label{ass2}
\begin{enumerate}
\item The support $T$ of $\beta$ is random and has uniform distribution among all index subsets of 
$\left\{ 1,\ldots,n\right\}$ with cardinal $s$,
\item Given $T$, the sign pattern of $\beta_T$ is random with uniform distribution over $\{-1,+1\}^{s}$, and jointly independent of the support.
\end{enumerate}
\end{ass}
\ 

The last condition concerns the magnitude of the nonzero regression coefficients $\beta_j$, $j\in T$. 
Let $\alpha>0$, $r\in(0,\frac12]$ and 
\bea \label{kappa}
\kappa & =& 4\sqrt{1+\alpha}.
\eea
Let us now define 
\bea
\mathsf{H}_{\alpha,r}^{n,s_0,p} &= & 
4 \frac{\sqrt{n}+\sqrt{2\alpha \log p}}{\sqrt{s_0}} \frac{1-r}{\sqrt{1+r}}.
\label{mH}
\eea
Let us introduce the function
\bean
\ell_\alpha(x) & = & x e^{-4\alpha / x}, \quad x>0.
\eean
Since $\ell_\alpha((0,+\infty))=(0,+\infty) $, the following constant $C_\circ:=C_\circ(\alpha,r)$ is well defined:
\bea
\label{Ccirc}
\ell_\alpha(C_\circ) & = & 10 e \frac{1+r}{(1-r)^2} \kappa^2 >0.
\eea
It will appear in the number $n$ of observations (explaining the index '$\circ$').\\

We can now define the range assumption for the coefficients of $\beta$ for Strategy (A).\\

\begin{ass}{\rm (Range condition for $\beta$: Strategy (A))}
\label{ass3high}
The unknown vector $\beta$ verifies
\bea
\label{minbetH}
\min_{j\in T}|\beta_{j}| &\ge & \mathsf{H}_{\alpha,r}^{n,s_0,p} \ \sg.
\eea
\end{ass}


Our main results show that the estimators $\ch \beta$ defined by either Strategy (A) or 
Strategy (B) recovers the support and sign pattern of $\beta$ exactly with probability of the order $1-O(p^{-\alpha})$
using similar bounds on the coherence and the sparsity as in \cite{CandesPlan:AnnStat09}. \\

As from now, let us choose $r\in(0,\frac12]$ and set:
\bea
C_{\rm spar}& = & \frac{r^2}{(1+\alpha)e^2} \\
C_\mu & = & \frac{r}{1+\alpha}.
\eea

\begin{theo}\label{main2}
Let $\alpha >0$ and $ p\ge e^{8/\alpha}$. Let Assumption \ref{ass1} hold with $C_\mu$ given above. Let Assumptions \ref{ass2} and \ref{ass3high} hold with
\bea
s & \le & s_0:=\frac{p}{\log p} \ \frac{C_{\rm spar}}{\|X\|^{2}} \label{ups}
\label{s}\\
n & \ge & s \left(C_\circ \log p +1\right) \label{lown}.
\eea
Then 
the probability that the estimator 
$\ch \beta$ defined by Strategy (A) with 
\bea \label{int-cvar}
C_{{\rm var}} & \in & \left[ \frac{(1-r)^2}{20(1+r)C_{\rm spar}}\ \frac{n}{p} \|X\|^{2} ; \ \frac{(1-r)^2}{2(1+r)C_{\rm spar}}\ \frac{n}{p} \|X\|^{2}\right],
\eea
 exactly recovers the support and sign pattern of 
$\beta$ is greater than $ 1-228/p^{\alpha}$.
\end{theo}
\ 

\begin{rem}
The choiCe of the constant 20 is unessential and the reader can check for himself which range 
is relevant for his own specific application.
\end{rem}

We now turn to Strategy (B).
Let us define for $C>0$,
\bea
\mathsf{L}_{\alpha,r,C}^{n,s,p} &= &  \label{mL} 
\max\left(  2\frac{\sqrt{1+2C}}{C\sqrt{1-r}} \frac{\sqrt{n-s}+\sqrt{2 \alpha \log p}}{ \sqrt{s}} \ ,
2 \frac{\sqrt{s}+\sqrt{2 \alpha \log p}}{\sqrt{1-r}\ \sqrt{s}} \right)\\
\mathsf{M}_{\alpha,r,C}^{n,s,p}  &=& \frac{n-s}{\sqrt{\log p}}\ \frac{1}
{3 \kappa C} \left(\frac{\sqrt{\pi (n-s)}}{p^{\alpha}} \right)^{\frac{4}{n-s}}.    \label{M} 
\eea

Let us state the corresponding range assumption for the coefficients of $\beta$.
\begin{ass}{\rm (Range condition for $\beta$: Strategy (B))}
\label{ass3low}
The unknown vector $\beta$ verifies
\bea
\label{minbetL}
\min_{i\in T}|\beta_{j}| &\ge & \mathsf{L}_{\alpha,r,C}^{n,s,p} \ \sg,\\
\label{maxbetL}
\|\beta\|_1 & \le & \mathsf{M}_{\alpha,r,C}^{n,s,p}  \ \sg.
\eea
\end{ass}

\bigskip

\begin{theo}\label{main3}
Let $\alpha >0$, $ p\ge e^{8/\alpha}$ and set $c_\circ  = \frac{(6\kappa)^2e}{1-r}$. Choose $C>0$. Let Assumptions \ref{ass1}, \ref{ass2} and \ref{ass3low} hold
with 
\bea
s & \le & \frac{p}{\log p} \ \frac{C_{\rm spar}}{\|X\|^{2}},
\label{ss}\\
n & \ge & c_\circ(1+2C) \ s \log p +s. \label{n-main3}
\eea
Then the probability that the estimator 
$\ch \beta$ defined by Strategy (B) exactly recovers the support and sign pattern of 
$\beta$ is greater than $ 1-229/p^{\alpha}.$
\end{theo}

\subsection{Important comments}

\subsubsection{About $X$}
\label{aboutX}
The normalized Gaussian example is instructive. 
First, when $X$ is obtained from a random matrix with i.i.d. standard Gaussian random entries by normalizing the columns, the coherence is of the order $\sqrt{\log p/n}$ (See below for a short proof). 
Therefore, taking $n$ of the order of $\log^3 p$ is sufficient for satisfying the Incoherence Assumption \ref{ass1}.
Second, it is also well known that $\|X\|^{2}$ is of the order $p/n$, see e.g. \cite {Rudelson:ICM}. 
This suggests in particular that the upper bound (\ref{ss}) on the number $s$ of nonzero components of $\beta$ may be understood in the Gaussian 
setting as  
\bean
s & \le & \frac{p}{\log p} \ \frac{C_{\rm spar}}{\|X\|^{2}}=O\left(\frac{n}{\log p}\right).
\label{s2}
\eean
This order of magnitude might be also valid for much more general random designs.  

Notice that the estimate $\sqrt{\log p/n}$ of the coherence for i.i.d. Gaussian matrices with normalized columns 
easily follows from the Paul Levy concentration of measure phenomenon on the sphere and 
the union bound: Since, due to normalization, each column is Haar distributed on the unit sphere,  
rotational invariance implies that the scalar product of two column vectors $X_j$ and $X_{j^\prime}$ satisfies 
\bean
\bP \left(\left|\la X_j,X_{j^\prime}\ra\right| \ge u \right) 
& = \bP \left(\left|\la X_j,X_{j^\prime} \ra\right| \ge u \mid X_{j^\prime} \right)
\le & 2 \exp\left(-c n \ u^2 \right),
\eean
for some constant $c$, by the well known concentration of measure phenomenon on the unit sphere. Thus, 
the union bound gives 
\bean
\bP \left(\max_{1\le j<j^\prime\le p}\left|\la X_j,X_{j^\prime} \ra\right| \ge u \right)
& \le & \frac{p(p-1)}2 \cdot 2 \exp\left(-c n \ u^2 \right), \\
& \le & \exp \left(-c n \ u^2 +2\log p \right).
\eean
This last quantity is less that $p^{-\alpha}$ for $u \ge c^\prime  \ \sqrt{\log p/n}$ with
$c^\prime= \sqrt{(\alpha+2)/c}$.

An interesting question concerns the pertinence of the coherence for the problem of variable selection using the 
LASSO. The work of \cite{Wainwright:IEEEIT09} shows through numerical investigations 
that certain conditions on the matrix $X$ (requiring in particular the knowledge of the true signal's support, without any statistical assumptions on beta though), 
allow to deduce sharp bounds on the minimum sample size needed for exact support recovery. When the 
true support is not known ahead of time, conditions such as the ones in \cite{Wainwright:IEEEIT09} 
are required to hold uniformly or at least for most support with high probability. Proving such 
a property for matrices more general than i.i.d. Gaussian matrices implies loosing sharp bounds on the minimum sample size. 
The advantage of the coherence over such assumptions relies in the fact that it can be computed very easily for 
any given matrix. The main drawback is that the resulting bounds on the minimum sample size might not be sharp.

\subsubsection{Order of $\mathsf{H}_{\alpha,r}^{n,s_0,p}$ }
In the case where $X$ is i.i.d. Gaussian, the order of $s_0$ is $n/\log p$ and thus the order of $\mathsf{H}_{\alpha,r}^{n,s_0,p}$
is $\sqrt{\log p}$, just as in \cite{CandesPlan:AnnStat09}. Indeed,
\bean
\mathsf{H}_{\alpha,r}^{n,s_0,p}
 & \asymp & 
 \frac{\sqrt{n}+\sqrt{2\alpha \log p}}{\sqrt{\frac{n}{\log p}}} \ \asymp \  \sqrt{\log p}.
\eean

\subsubsection{About $C$ and $\mathsf{L}_{\alpha,r,C}^{n,s,p}$}
\label{slide}
Increasing the upper bound on the magnitude of the $\beta_j$'s via decreasing the constant $C$ also results 
in increasing the lower bound. Therefore, $C$ governs a sliding window inside which the coefficients of $\beta$ 
can be recovered by the LASSO. Moreover for a given $n$, one can decrease the lower bound $\mathsf{L}_{\alpha,r,C}^{n,s,p}$ in Eq. (\ref{mL}) by increasing $C$. This would result on a smaller sparsity in Eq. (2.26).
Taking $C$ as $C \sim n/(s\log p)$ implies the usual order $\sqrt{\log p}$ for the minimum of beta's (See Eq. (\ref{mL})). If one wants to specify $C$ in a way that is independent of $s$ one may run the risk of prescribing an incorrect order for $\mathsf{L}_{\alpha,r,C}^{n,s,p}$ as a function of $n$. This technical issue should however be considered as of theoretical interest only and not so much of a problem in practice. As an analogy, consider the plain LASSO with known variance: there exists a universal way of choosing the parameter $\lb$, but many practitioners use the LARS instead in order to explore all the supports occuring on the $\lb$-trajectory and compare them using a standard model selection procedure (AIC, BIC, Foster and George, etc). In the same manner, one could also vary the value of $C$ and compare all supports on this trajectory. In this spirit,  our simulation experiments show the histogram of recovered and incorrectly detected components over a large range of values of $C$. One nice surprise is that Strategy (B) is quite robust vs. the actual choice of $C$ at such a low signal to noise ratio level.

\subsubsection{About the constants $C_{\rm spar}$ and $C_\mu$}

Let us compare the numerical values of these constants to the one obtained in \cite{CandesPlan:AnnStat09}.

One of the various constraints on the rate $\alpha$ in \cite{CandesPlan:AnnStat09} is given by the theorem of Tropp in \cite{Tropp:CRAS08}. In this setting,
\bean
\alpha & = & 2\log 2\\
r & = & 1/2,
\eean
the author's choice of $1/2$ being unessential.
To obtain such a rate $\alpha$, they need to impose the r.h.s. of (3.15) in \cite{CandesPlan:AnnStat09} to be less than $1/4$, that is:
\bea\label{C_cp}
30 C_{\mu} + 13 \sqrt{2C_{\rm spar}} & \le & \frac{1}{4}.
\eea
This yields $C_{\rm spar} <  1.19 \times 10^{-4}$. Let us choose $C_{\rm spar}$ close to this maximum allowed. Then, compute $C_{\mu}$ by (\ref{C_cp}). This yields
\bean
C_{\rm spar} \simeq  1.18\ 10^{-4},\quad
C_{\mu} \simeq  1.7\ 10^{-3}.
\eean
(The additional condition coming from the end of the proof of \cite[Lemma 3.5]{CandesPlan:AnnStat09}, that is $\frac{3}{64 C_\mu^2}=2\log(2)$, is not limiting since $\sqrt{3/(128 \log 2)} \gg 1.7\ 10^{-3}$.)
\smallskip

Our theorem allows to choose any rate $\alpha>0$. To make a fair comparison, let us also choose $\alpha  =  1.5>2\log 2$ and $r  =  1/2$.
We obtain:
\bean
C_{\rm spar}  \simeq 1.4\ 10^{-2}, \quad
C_{\mu}  =  0.2.
\eean

\section{Proof of Theorem \ref{main2}}
\label{Proofmain2}

The proof is divided into several steps. The main two steps are as follows. First, we provide the description and consequences of 
the optimality conditions for the standard LASSO estimator as a function of $\lb$. Second, we  
prove that these optimality conditions are satisfied by a simple and natural oracle estimator. 

\subsection{Enforcing the invertibility assumption}

We recall the basic result we proved in \cite{cd_inv} regarding the invertibility of random submatrices via the Non-commutative Chernoff Inequality.\\

\begin{theo}\label{main1}
Let $r\in(0,1)$, $\alpha \ge 1$. Let $X$ be a full-rank $n\times p$ matrix and $s$ be
positive integer, such that
\bean
\mu(X) & \le &  \frac{r}{2(1+\alpha)\log p}\\
s & \le & \frac{r^2}{4(1+\alpha) e^2}\ \frac{p}{ \|X\|^{2} \log p }.
\eean
Let $T\subset \left\{1,\ldots,p\right\}$ 
be a set with cardinality $s$, chosen randomly from the uniform distribution.
Then the following bound holds:
\bea \label{sing}
\bP \left(\|X_T^tX_T-\Id_s \|\ge r \right) & \le & \frac{216}{p^{\alpha}}.
\eea
\end{theo}

By Theorem \ref{main1}, we have
\bea
\label{b1}
(1+r)^{-1} \ \le \ \|\left(X_T^tX_T\right)^{-1}\| \ \le \ (1-r)^{-1}\\
\label{b2}
(1-r)^{1/2} \ \le \ \|X_T\| \ \le \ (1+r)^{1/2}
\eea
with probability greater than $1- 216 \ p^{-\alpha}$. Thus, throughout this section, we will assume that 
(\ref{b1}) and (\ref{b2}) hold, i.e. we will reduce all events considered to their intersection with 
the event that (\ref{b1}) and (\ref{b2}) are satisfied.

\subsection{The oracle estimator for $\ch \beta$ and $\ch \lb$}
\label{oraclesig}
We now discuss the next step of the proof of Theorem \ref{main2}, which consists of studying 
some sort of oracle estimators for $\beta$ which enjoys the property of knowing the support $T$ of $\beta$
ahead of time. 

For a given $\wdt{\lambda}$, one might like to consider the following oracle for $\ch{\beta}$:
\bea
\overline{\beta} & \in & \down{b\in \cal B}{\rm argmax}\ -\frac12 \|y-Xb\|_2^2 - \wdt{\lambda}\|b\|_1,
\eea
where
\bean
\cal B & =& \{b\in \mathbb R^p,\: {\rm supp}(b)=T,\ \sgn(b)=\sgn(\beta_{T}) \}.
\eean
However, it is not so easy to derive a closed form expression for $\overline{\beta}$.
Therefore, it might be more interesting to consider instead the following oracle:
\bea
\wdt{\beta} &\in &  \down{b\in \R^p,\: {\rm supp}(b)=T}{\rm argmax}
-\frac12 \|y-Xb\|_2^2 - \wdt{\lambda} \ \sgn(\beta_{T})^t b.
\eea
Indeed, $\wdt{\beta}$ satisfies
\bean
X_T^t\left(y-X_T\wdt{\beta}_T\right) - \wdt{\lambda} \: \sgn(\beta_{T}) & = & 0,
\eean
and we obtain that $\wdt{\beta}$ is given by
\bea
\label{tildebet}
\wdt{\beta}_{T}  & = & \left(X_T^tX_T\right)^{-1}\left(X_T^t y - \wdt{\lambda} \ \sgn(\beta_{T}) \right).
\eea
This formula is the same as in the proof of Th. 1.3 in \cite{CandesPlan:AnnStat09}, but here, $\wdt \lb$ is a variable.\\

Now let us recall that in the known variance case, Cand\`es and Plan assume that 
\bea
\lb^2 & = & C_{{\rm var}} \sg^2 \log p,
\eea 
for some positive constant $C_{{\rm var}}$. It is then relevant to seek our oracle $\wdt \lb$ as:
\bea \label{condlb}
\wdt\lambda^2 & = &  C_{{\rm var}} \ \frac{\|y-X_T \wdt{\beta}_T\|_2^2}{n} \ \log p.
\eea
Replacing $\wdt{\beta}$ by its value (\ref{tildebet}), we obtain
\bean
C_{{\rm var}} \ \|y-X_T  \left(X_T^tX_T\right)^{-1} \left(X_T^ty -\wdt{\lambda} \ \sgn(\beta_{T}) \right) \|_2^2 &  = & \frac{n}{\log p} \ \wdt\lambda^2.
\eean
Thus,
\bean
C_{{\rm var}} \ \| \gP_{V_T^\perp}y + \wdt{\lambda}  X_T  \left(X_T^tX_T\right)^{-1}\sgn(\beta_{T}) \|_2^2
&  = & \frac{n}{\log p} \ \wdt\lambda^2,
\eean
and using the orthogonality relations, we obtain
\bean
C_{{\rm var}} \ \|\gP_{V_T^\perp}y\|_2^2 +\wdt{\lambda}^2 C_{{\rm var}} \ \|X_T  \left(X_T^tX_T\right)^{-1} \sgn(\beta_{T}) \|_2^2 &  = & \frac{n}{\log p} \ \wdt\lambda^2,
\eean
which is equivalent to
\bea
\label{lambtilde1}
\wdt{\lb}^2 
 & = &
\frac{  \|\gP_{V_T^\perp}z\|_2^2 }{\frac{n}{ C_{{\rm var}} \log p} - \ \|X_T  \left(X_T^tX_T\right)^{-1} \sgn(\beta_{T}) \|_2^2} 
\eea
We henceforth work with this definition of $\wdt{\lb}$.
Notice that $\wdt{\lb}$ is well defined whenever 
\bea \label{cond_up_Clb}
C_{{\rm var}}  & \le & \frac{n}{\|X_T  \left(X_T^tX_T\right)^{-1} \sgn(\beta_{T}) \|_2^2 \ \log p}.
\eea
The choice of $C_{{\rm var}}$ will be done in the next section.

\subsection{Study of the oracle $\wdt{\lb}$}
\label{lamblowsnr}

In this section, we provide a confidence interval for $\wdt{\lb}$. In particular, the first subsection shows that $\wdt{\lb}$ is well defined.

\subsubsection{Bounds on $\|X_T  \left(X_T^tX_T\right)^{-1} \sgn(\beta_{T}) \|_2^2$}
Using the lower bound on $\sg_{\min}(X_T)$ and the upper bound on $\sg_{\max}(X_T)$ given by (\ref{b1}) and (\ref{b2}), we have, with high probability:
\bea
\label{den}
\frac{1-r}{(1+r)^2} \ s \le & \|X_T  \left(X_T^tX_T\right)^{-1} \sgn(\beta_{T}) \|_2^2 & \le \frac{1+r}{(1-r)^2} \ s.
\eea 
We write the choice of $C_{{\rm var}}$ made in (\ref{int-cvar}) as
\bea \label{cvar}
\frac1{20} \frac{(1-r)^2}{1+r} \frac{n}{s_0 \ \log p} \ \le\  C_{{\rm var}}  \ \le \ \frac12 \frac{(1-r)^2}{1+r} \frac{n}{s_0 \ \log p},
\eea 
where $s_0$ is the maximum sparsity allowed in Inequality (\ref{ups}), namely,
\bean
s_0 & = & \frac{p}{\log p} \ \frac{C_{\rm spar}}{\|X\|^{2}}.
\eean
In particular, the condition (\ref{cond_up_Clb}) is satisfied which garantees that $\wdt{\lb}$ is 
indeed well defined.

\subsubsection{Bounds on $\| \gP_{V_T^\perp}z\|_2$}
\label{VTperp}
Using some well known properties of the $\chi^2$ distribution recalled in Lemma \ref{cki} in the Appendix, 
we obtain that
\bea
\label{tperp}
\bP\left(\|\gP_{V_T^\perp}(z)\|_2/\sigma \ge \sqrt{n-s}+ \sqrt{2t} \right) & \le & \exp(-t)
\eea
and
\bea
\label{u}
\bP\left(\|\gP_{V_T^\perp}(z)\|_2^2/\sigma^2  \le u (n-s)\right)  & \le & 
\frac2{\sqrt{\pi (n-s)}}\left(u\ e/2\right)^{\frac{n-s}4}.
\eea
Tune $u$ such that the r.h.s. of (\ref{u}) equals $2/p^{\alpha}$, i.e.
\bean
u & = & \frac{2}{e} \left(\frac{ \sqrt{\pi (n-s)}}{p^{\alpha}} \right)^{4/(n-s)}.
\eean
Thus, we obtain that
\bea
\label{major}
\frac{\|\gP_{V_T^\perp}z\|_2^2}{\sigma^2} \le  \left(\sqrt{n-s}+\sqrt{2 \log(\frac{p^{\alpha}}{2})}\right)^2
\le \left(\sqrt{n-s}+\sqrt{2\alpha \log p}\right)^2
\eea
and
\bea
\label{minor}
\frac{\|\gP_{V_T^\perp}z\|_2^2}{\sigma^2}  &  \ge &  \frac{2(n-s)}{e} \left(\frac{\sqrt{\pi (n-s)}}{p^{\alpha}} \right)^{4/(n-s)}
\eea
with probability greater than or equal to $1-2p^{-\alpha}$.

\subsubsection{Bounds on $\tilde{\lb}$}

\begin{lemm}
The following bounds hold:
\bea
\label{up}
\wdt{\lambda} & \le &  \sg \ \frac{1-r}{\sqrt{1+r}} \ \frac{\sqrt{n-s}+\sqrt{2\alpha \log p}}{\sqrt{s_0}} \\
\wdt{\lambda} & \ge & \kappa \ \sg \ \sqrt{\log p}. \label{low}
\eea
\end{lemm}

\begin{proof}
Recall that $0\le s \le s_0$. 
From (\ref{cvar}), we have
\bean
C_{{\rm var}}  \ \le \ \frac12 \frac{(1-r)^2}{1+r} \frac{n}{s_0 \ \log p}.
\eean
We then obtain, by virtue of (\ref{lambtilde1}) and the upper bound in (\ref{den}), 
\bean
\wdt{\lambda}^2 & \le &  \frac{\|\gP_{V_T^\perp}z\|_2^2 }{2s_0 \frac{1+r}{(1-r)^2}- \|X_T  \left(X_T^tX_T\right)^{-1} 
\sgn(\beta_{T}) \|_2^2} \\
 & \le &  \frac{\|\gP_{V_T^\perp}z\|_2^2 }{2s_0 \frac{1+r}{(1-r)^2}- s_0 \frac{1+r}{(1-r)^2}}.
\eean
Using the bound (\ref{major}), we deduce (\ref{up}).\\

On the other hand, the bound (\ref{minor}) and 
\bean
\frac{n}{C_{{\rm var}}\log p} & \le & 20 \frac{1+r}{(1-r)^2} s_0,
\eean
yield
\bean
\label{cond_low_Clb}
\wdt{\lambda}^2 & \ge & \frac{2(n-s)}{e} \left(\frac{\pi (n-s)}{p^{2\alpha}} \right)^{2/(n-s)}
\frac{\sg^2}{20 \frac{1+r}{(1-r)^2}s_0}.
\eean
From (\ref{lown}), we know that $n$ verifies
\bea\label{n-s}
\frac{n-s}{s_0} \ge \frac{n-s_0}{s_0} \ge C_\circ \log p.
\eea
Thus, noting that $(\pi(n-s))^{2/(n-s)}\ge 1$, 
\bean
\wdt{\lambda}^2 & \ge &  \frac{(1-r)^2}{10e(1+r)} p^{-4\alpha/(n-s)} C_\circ \sg^2 \log p.
\eean
Writing $p^{-4\alpha/(n-s)}=e^{-4\alpha \log p/(n-s)}$ and, using (\ref{n-s}) again,
\bean
\frac{\log p}{n-s} \le \frac{\log p}{n-s_0}  \le  \frac{1}{s_0\ C_\circ} \le \frac{1}{C_\circ},
\eean
we obtain
\bean
p^{-4\alpha/(n-s)} & \ge  & e^{-4\alpha/ C_\circ}.
\eean
Therefore, 
\bea
\wdt{\lambda}^2 & \ge & \ \frac{(1-r)^2}{10e(1+r)} e^{-4\alpha/ C_\circ} C_\circ \sg^2 \log p.
\eea
Let us recall that the constant $C_\circ$ has been precisely chosen to satisfy
\bean
\ell_\alpha(C_\circ) = C_\circ e^{-4\alpha/ C_\circ} = 10 e\ \frac{1+r}{(1-r)^2}  \kappa^2.
\eean
As a conclusion, we have just proved (\ref{low}).
\end{proof}

\subsection{Cand\`es and Plan's conditions}
\label{CPcond}

To obtain the exact recovery of the support and sign patterns of $\beta$, we will need similar bounds as the ones 
in \cite[Section 3.5]{CandesPlan:AnnStat09}. Namely,
\begin{enumerate}
\label{assump}
\item[(i)]  $\|(X_T^t X_T)^{-1}X_T^t z\|_{\infty}\le \kappa \ \sigma \sqrt{\log p}$
\smallskip
\item[(ii)]  $\|(X_T^t X_T)^{-1}\sgn(\beta_T)\|_{\infty}\le 3$
\smallskip
\item[(iii)]  $\|X_{T^c}^t X_T (X_T^t X_T)^{-1}\sgn(\beta_T)\|_{\infty}\le \frac14$
\smallskip
\item[(iv)]  $\|X_{T^c}^t \left(\Id -X_T (X_T^t X_T)^{-1}X_T^t \right) z\|_{\infty}\le \kappa
\ \sg \sqrt{\log p}$
\smallskip
\item[(v)]  $\|X_T^tX_T-\Id_s \| \le r$.
\end{enumerate}

When $r=\frac12$, these conditions were proven to hold with high probability in \cite{CandesPlan:AnnStat09} based on previous results due to Tropp \cite{Tropp:CRAS08}. 
Most of the proofs that these conditions hold with high probability are the same as in \cite{CandesPlan:AnnStat09} up to some slight improvements of the constants. \\

\begin{prop}
\label{linfs}
The bounds (i-iv) hold with probability at least $1-10/p^\alpha$. Condition (v) holds  with probability at least $1-216/p^\alpha$.
\end{prop}

\begin{proof}
See Section \ref{cp-cond} in the Appendix. 
\end{proof}  

\subsection{Last step of the proof}

We now conclude the proof using the strategy announced in the beginning of this section:
\begin{enumerate}
\item[(i)] We prove that the proxies $\wdt{\beta}$ and $\wdt{\lb}$ satisfy
the optimality conditions (\ref{cond1}) and (\ref{cond2}), from which we deduce that $\ch\beta =\wdt \beta$ and $\ch \lb=\wdt \lb$.
\item[(ii)] Since the proxy $\wdt \beta$ has the right support and sign patterns, we conclude that $\ch \beta$ exactly recovers these features as well.
\end{enumerate}

\subsubsection{$\wdt{\beta}$ and $\beta$ have the same support and sign pattern} \label{stratAbeta}

First, it is clear that $\wdt{\beta}$ and $\beta$ have the same support. Next,
we must prove that $\wdt{\beta}$ has the same sign pattern as $\beta$.
Use Proposition \ref{linfs} to obtain
\bean
\|\wdt{\beta}_T-\beta_T\|_{\infty} & \le &
\|(X_T^t X_T)^{-1}X_T^tz\|_{\infty}+\wdt \lb \ \|(X_T^t X_T)^{-1}\sgn(\beta_T)\|_{\infty} \\
& \le & \kappa \ \sigma \sqrt{\log p}+ 3 \wdt \lb.
\eean
Using the lower bound (\ref{low}), 
and the expression of $\kappa$, we obtain
\bea
\|\wdt{\beta}_T-\beta_T\|_{\infty} & \le & 4 \wdt \lb.
\eea
A sufficient condition to guarantee that the sign pattern is recovered is that this last upper bound 
be lower than the minimum absolute value of non-zero components of $\beta$, i.e.
\bea
\label{nonintersect}
4 \wdt \lb & \le & \min_{j\in T} |\beta_j|.
\eea
Using the upper bound on $\wdt \lb$ in (\ref{up}), 
this is achieved in particular when
\bean
4 \sg \ \frac{\sqrt{n-s}+\sqrt{2 \alpha \ \log(p)}}{\sqrt{s_0}} \frac{1-r}{\sqrt{1+r}} & \le & \min_{j\in T} |\beta_j|,
\eean
which is implied by Assumption \ref{ass3low}.

\subsubsection{$\wdt{\beta}$ and $\wdt\lb$ satisfy the optimality conditions}
\label{oracleoptlow}
Using the lower bound (\ref{low}) on $\wdt \lb$, the proof of the fact
$\wdt{\beta}$ and $\wdt\lb$ satisfy the optimality conditions is exactly
the same as in \cite[Section 3.5]{CandesPlan:AnnStat09}. We repeat the argument for the
sake of completeness. On one hand, by construction, we clearly have
\bean
X_T^t(y-X\wdt{\beta}) & = & -\wdt \lb \ \sgn(\beta_T).
\eean
Since $\wdt{\beta}$ and $\beta$ have the same sign pattern, we actually have:
\bean
X_T^t(y-X\wdt{\beta}) & = & -\wdt \lb \ \sgn(\wdt \beta_T).
\eean
On the other hand,
\bea
\|X_{T^c}^t(y-X\wdt{\beta})\|_\infty & = & \|X_{T^c}^t\gP_{V^\perp}(z)
+\wdt \lb \ X_{T^c}^t X_T (X_T^t X_T)^{-1}\sgn(\beta_T) \|_\infty \nonumber \\
& \le & \|X_{T^c}^t\gP_{V^\perp}(z)\|_\infty
+\wdt \lb \ \|X_{T^c}^t X_T (X_T^t X_T)^{-1}\sgn(\beta_T) \|_\infty \nonumber \\
& \le & \kappa \ \sg\sqrt{\log p} +  \frac14 \wdt \lb \label{kap} \\
& \le & \  \frac{3}{4} \ \wdt  \lb \ <  \ \wdt \lb. \nonumber 
\eea

Hence, the two parts of the subgradient conditions (\ref{cond1}-\ref{cond2}) are satisfied by $\wdt{\beta}$ and $\wdt\lb$, which means that 
\bea \label{cool}
\wdt \beta & = & \ch \beta_{\wdt \lb}.
\eea
In other words, $\wdt \beta$ corresponds to the solution of problem (\ref{lass}) with the penalization $\lb=\wdt \lb$. Moreover, $\wdt \lb$ bas been determined so that it verifies (\ref{condlb})
\bean
\wdt\lambda^2 & = &  C_{{\rm var}} \ \frac{\|y-X_T \wdt{\beta}_T\|_2^2}{n} \ \log p,
\eean
i.e., plugging (\ref{cool}),
\bean
\wdt\lambda^2 & = &  C_{{\rm var}} \ \frac{\|y-X_T (\ch \beta_{\wdt \lb})_T\|_2^2}{n} \ \log p.
\eean
Therefore, $\wdt \lb$ is a solution of Eq. (\ref{eqA}). By virtue of uniqueness proved in Appendix \ref{sec-gammaA},
we deduce that 
\bean
\ch{\beta} & = & \wdt{\beta}\\
\ch{\lambda} & = & \wdt\lb.
\eean

\subsubsection{Conclusion of the proof}  \label{concl-theo1}
The two preceding sub-sections prove that $\ch{\beta}$ has same support and sign pattern as
$\beta$. This occurs when (\ref{b1}) and (\ref{b2}) (both implied by the invertibility condition (v) in Sec. \ref{CPcond}), Cand\`es and Plan's conditions (i-iv) in Sec. \ref{CPcond} and 
the bound on $\| \gP_{V_T^\perp}z\|_2$ in Sec. \ref{VTperp} are satisfied simultaneously. 
Therefore, this occurs with probability at least
\bean
1-\frac{216 + 10 + 2}{p^{\alpha}},
\eean
as announced.

\section{Proof of Theorem \ref{main3}}
\label{Proofmain3}

As in the proof of Theorem \ref{main2}, the quasi-isometry property (\ref{b1}) and (\ref{b2}), and Cand\`es and Plan's conditions
of Section \ref{CPcond} will be assumed. Notice also that 
the results of Section \ref{optlamb} are still valid with the assumption of Theorem \ref{main3}. 

\subsection{The oracle estimator}


As in the case of Section \ref{oraclesig}, the oracle for $\beta$ is given by 
\bea
\label{tildebet2}
\wdt{\beta}_{T}  & = & \left(X_T^tX_T\right)^{-1}\left(X_T^t y - \wdt{\lambda} \ \sgn(\beta_{T}) \right).
\eea

We now seek $\wdt \lb$ verifying
\bea \label{condlb2}
\frac12 \|y-X_T \wdt{\beta}_T\|_2^2 & = & C \wdt\lambda\ \sgn(\beta_{T})^t \wdt{\beta}_T.
\eea
Replacing $\wdt{\beta}$ by its value (\ref{tildebet}), we obtain
\bean
& \d \frac12 &\|y-X_T  \left(X_T^tX_T\right)^{-1} \left(X_T^ty -\wdt{\lambda} \ \sgn(\beta_{T}) \right) \|_2^2 \\
& & \quad = \ C\  \wdt{\lambda} \ \sgn(\beta_{T})^t \left(\left(X_T^tX_T\right)^{-1} 
\left(X_T^ty -\wdt{\lambda} \ \sgn(\beta_{T}) \right)\right).
\eean
Thus,
\bean
\frac12\  \| \gP_{V_T^\perp}y + \wdt{\lambda}  X_T  \left(X_T^tX_T\right)^{-1}\sgn(\beta_{T}) \|_2^2
\  = \hspace{2cm}  \\
 - C \wdt{\lambda}^2 \la \sgn(\beta_{T}),\left(X_T^tX_T\right)^{-1}\sgn(\beta_{T}) \ra 
+C \wdt{\lambda} \ \sgn(\beta_{T})^t \left(X_T^tX_T\right)^{-1}X_T^ty.
\eean
Using the orthogonality relations, we then obtain
\bean
\frac12 \|\gP_{V_T^\perp}y\|_2^2  +\frac{\wdt{\lambda}^2}{2} \|X_T  \left(X_T^tX_T\right)^{-1} \sgn(\beta_{T}) \|_2^2  =  C \wdt{\lambda} \ \sgn(\beta_{T})^t \left(X_T^tX_T\right)^{-1}X_T^ty\\
 \hspace{.3cm} - C\: \wdt{\lambda}^2 \la \sgn(\beta_{T}),\left(X_T^tX_T\right)^{-1}\sgn(\beta_{T}) \ra,
\eean
which is equivalent to
\bea
\label{quad} \left(\frac{1}{2}  +C\right) \wdt{\lambda}^2 \|\left(X_T^tX_T\right)^{-\frac12} \sgn(\beta_{T}) \|_2^2 - C \wdt{\lambda} \ \sgn(\beta_{T})^t \left(X_T^tX_T\right)^{-1}X_T^ty+\frac12 \|\gP_{V_T^\perp}z\|_2^2 = 0. 
\eea
The roots of the quadratic equation are
\bea
\label{lambtilde}
\wdt{\lambda} & = &
\frac{C\: \sgn(\beta_{T})^t \left(X_T^tX_T\right)^{-1}X_T^ty\pm \sqrt{\Delta} }{(1+2C) \|\left(X_T^tX_T\right)^{-\frac12} \sgn(\beta_{T}) \|_2^2},
\eea
where
\bean
\Delta & = &\left(C \: \sgn(\beta_{T})^t \left(X_T^tX_T\right)^{-1}X_T^ty\right)^2 \\
& & - (1+2C) \|\left(X_T^tX_T\right)^{-\frac12} \sgn(\beta_{T}) \|_2^2\|\gP_{V_T^\perp}z\|_2^2.
\eean

\subsection{Study of the oracle $\tilde{\lb}$}
\label{equalC}
Following the same strategy as for Strategy (A), we now provide a confidence interval for $\tilde{\lb}$. 

\subsubsection{Premilinaries}
We have
\bean
\sgn(\beta_{T})^t \left(X_T^tX_T\right)^{-1}X_T^t y & = & \sgn(\beta_{T})^t \left(X_T^tX_T\right)^{-1}X_T^t 
(X_T\beta+z) \\
& = & \sgn(\beta_{T})^t\beta+\sgn(\beta_{T})^t \left(X_T^tX_T\right)^{-1}X_T^t z \\
& = & \|\beta\|_1 +\langle X_T \left(X_T^tX_T\right)^{-1}\sgn(\beta_{T}),\gP_{V_T}z+\gP_{V_T^\perp}z\rangle.
\eean
Hence,
\beq \label{delta1}
\sgn (\beta_{T})^t  \left(X_T^tX_T\right)^{-1} X_T^t y \\
 =  \|\beta\|_1+\langle X_T \left(X_T^tX_T\right)^{-1}\sgn(\beta_{T}), \gP_{V_T}z\rangle.
\eeq
Note that the Cauchy-Schwarz inequality yields
\beq \label{cs1}
\left| \langle  X_T \left(X_T^tX_T\right)^{-1}\sgn(\beta_{T}) ,\gP_{V_T}z\rangle \right| 
\le \|\left(X_T^tX_T\right)^{-\frac12}\sgn(\beta_{T})\|_2 \| \gP_{V_T}z\|_2.
\eeq

\subsubsection{Bound on $\| \gP_{V_T}z\|_2$}
\label{VT}
Using some well known properties of the $\chi^2$ distribution recalled in Lemma \ref{cki} in the Appendix, 
we obtain 
\bea
\label{t}
\bP\left(\|\gP_{V_T}(z)\|_2/\sigma \ge \sqrt{s}+ \sqrt{2t} \right) &\le & \exp(-t).
\eea
Tune $t$ such that $e^{-t}=2 p^{-\alpha}$, i.e.
\bean
t & = & \log(p^{\alpha}/2).
\eean
Hence,
\bea
\label{tprime}
\bP\left(\|\gP_{V_T}(z)\|_2/\sigma \ge \sqrt{s}+ \sqrt{2\log(p^{\alpha}/2)} \right) &\le & p^{-\alpha}. 
\eea

\subsubsection{Positivity of $\Delta$}\

We begin with the study of $\sgn(\beta_{T})^t \left(X_T^tX_T\right)^{-1}X_T^ty$ and $\|\left(X_T^tX_T\right)^{-\frac12} \sgn(\beta_{T}) \|_2^2\|\gP_{V_T^\perp}z\|_2^2$, two key quantities in the analysis.\\

We first study $\sgn(\beta_{T})^t \left(X_T^tX_T\right)^{-1}X_T^ty$. By (\ref{b1}), we have 
\bea \label{xxg}
\|\left(X_T^tX_T\right)^{-\frac12}\sgn(\beta_{T})\|_2 & \le & \sqrt{\frac{s}{1-r}}.
\eea
Thus, using (\ref{tprime}), (\ref{cs1}) and the lower bound (\ref{minbetL}) from Assumption \ref{ass3low}  on the non-zero components of $\beta$, we can write
\bean
\left| \langle X_T \left(X_T^tX_T\right)^{-1}\sgn(\beta_{T}) ,\gP_{V_T}z\rangle \right| & \le & \sg \frac{\sqrt{s}}{\sqrt{1-r}} \left( \sqrt{s}+\sqrt{2\alpha \log p} \right) \\
  & \le & \frac12 \ \|\beta\|_{1}.
\eean

Therefore, from (\ref{delta1}) we deduce that
\bea
\label{downup}
\frac12 \|\beta\|_1
\le
\ \sgn(\beta_{T})^t \left(X_T^tX_T\right)^{-1}X_T^ty
\le
\frac32 \|\beta\|_1.\\
\nonumber
\eea

Second, we study $\|\left(X_T^tX_T\right)^{-\frac12} \sgn(\beta_{T}) \|_2^2
\|\gP_{V_T^\perp}z\|_2^2$. We have
\bean
\|\left(X_T^tX_T\right)^{-\frac12} \sgn(\beta_{T}) \|_2
\|\gP_{V_T^\perp}z\|_2  
&\le &\sg  \sqrt{\frac{s}{1-r}} \left(\sqrt{n- s}+\sqrt{2\alpha \log p} \right).
\eean

Thus 
\bea
\Delta & \ge & \frac{C^2}{4} \|\beta\|_1^2 - \sg^2 (1+2C)  \frac{s}{1-r} \left(\sqrt{n- s}+\sqrt{2\alpha \log p} \right)^2 \\
& \ge & \frac{C^2}{4} s^2 \min_{1\le j\le p}|\beta_j|^2 - \sg^2 (1+2C)  \frac{s}{1-r} \left(\sqrt{n- s}+\sqrt{2\alpha \log p} \right)^2
\eea
and Assumption \ref{ass3low} shows that $\Delta>0$, which ensures that  $\wdt{\lb}$ is well defined.\\

\subsubsection{Bounds on $\tilde{\lb}$}
First, let us write
\bean
\sqrt{\Delta} =\left(C \: \sgn(\beta_{T})^t \left(X_T^tX_T\right)^{-1}X_T^ty\right) \hspace{4cm} \\
\times \sqrt{1-\frac{(1+2C)\ \|\left(X_T^tX_T\right)^{-\frac12} \sgn(\beta_{T}) \|_2^2
\|\gP_{V_T^\perp}z\|_2^2}{\left(C \: \sgn(\beta_{T})^t \left(X_T^tX_T\right)^{-1}X_T^ty\right)^2}}.
\eean
On one hand, due to $\sqrt{1-\delta}\le 1-\frac{\delta}{2}$ on $(0,1)$,
we obtain
\bean
\sqrt{\Delta} & \le &
\left(C \: \sgn(\beta_{T})^t \left(X_T^tX_T\right)^{-1}X_T^ty\right) \\
& & -\frac{(1+2C) \|\left(X_T^tX_T\right)^{-\frac12} \sgn(\beta_{T}) \|_2^2
\|\gP_{V_T^\perp}z\|_2^2}{ 2C \: \sgn(\beta_{T})^t \left(X_T^tX_T\right)^{-1}X_T^ty}.
\eean
Combining this last equation with (\ref{lambtilde}), we obtain that
\bea
\wdt{\lambda} & \ge &
\frac{\|\gP_{V_T^\perp}z\|_2^2}
{2C\: \sgn (\beta_{T})^t \left(X_T^tX_T\right)^{-1}X_T^ty}.
\eea
On the other hand, we also have $\sqrt{1-\delta}\ge 1- \delta$ on $(0,1)$.
Thus we can write
\bean
\sqrt{\Delta}
& \ge & \left(C \: \sgn(\beta_{T})^t \left(X_T^tX_T\right)^{-1}X_T^ty\right) \\
& & -\frac{(1+2C)  \|\left(X_T^tX_T\right)^{-\frac12} \sgn(\beta_{T}) \|_2^2
\|\gP_{V_T^\perp}z\|_2^2}{C \: \sgn(\beta_{T})^t \left(X_T^tX_T\right)^{-1}X_T^ty}
\eean
and combining this last equation with (\ref{lambtilde}) and the previous upper bound, we thus obtain
\bean
\wdt{\lambda} & \le &
\frac{\|\gP_{V_T^\perp}z\|_2^2}
{ C\: \sgn(\beta_{T})^t \left(X_T^tX_T\right)^{-1}X_T^ty}.
\eean
Using (\ref{downup}), we finally get
\bea \label{boundC<1/2} 
\frac{\|\gP_{V_T^\perp}z\|_2^2}
{3\ C \: \|\beta\|_1} & \le \
\wdt{\lambda}\  \le &
2 \frac{\|\gP_{V_T^\perp}z\|_2^2}
{C \: \|\beta\|_1}.\\
 & & \nonumber
\eea
Combining this last equation with (\ref{minor}), we obtain:
\bea
\label{sandwC}
\sigma^2  \frac{(n-s) \left(\frac{\sqrt{\pi (n-s)}}{p^{\alpha}} \right)^{\frac{4}{n-s}}
}{3\ C\ 
\|\beta\|_1}\: \le \ \wdt \lb \ \le \
2\ \sigma^2\ \frac{\left(\sqrt{n-s}+\sqrt{2\alpha \log p}\right)^2}{C \|\beta\|_1}.
\eea
Using Assumption \ref{ass3low} and (\ref{M}), we thus obtain
\bea
\label{lowC}
\wdt \lb  & \ge & \kappa \ \sg \sqrt{\log p}.
\eea

\subsection{Last step of the proof}
\subsubsection{$\wdt{\beta}$ and $\beta$ have the same support and sign pattern}
As in the case of Strategy (A) it is clear that $\tilde{\beta}$ and $\beta$ 
have the same support. Let us now verify that they have the same sign pattern. 

As in Section \ref{stratAbeta} and based on (\ref{lowC}), we obtain
\bean
\|\wdt{\beta}_T-\beta_T\|_{\infty} & \le & 4 \wdt \lb,
\eean
exactly as for Strategy (A).
Using the upper bound on $\wdt \lb$ in the right hand side of (\ref{sandwC}), 
we thus need 
\bean
8\  \frac{\left(\sqrt{n-s}+\sqrt{2\alpha \log p}\right)^2}
{C}& \le & \min_{j\in T} |\beta_j| \ \frac{\|\beta\|_1}{\sigma^2}
\eean
to garantee that $\wdt{\beta}_T$ and $\beta_T$ have the same sign pattern.
In view of this inequality, and since $\|\beta\|_1\ge s \min_{j\in T} |\beta_j|$, an even stronger sufficient condition is
\bean
8\  \frac{\left(\sqrt{n-s}+\sqrt{2\alpha \log p}\right)^2}{C\ s }& \le & 
\frac{\min_{j\in T} |\beta_j|^{2}}{\sigma^2}.
\eean
Noting that $\frac{2\sqrt{2}}{\sqrt{C}}\le 2\frac{\sqrt{1+2C}}{C\sqrt{1-r}}$, we conclude that this condition is also implied by Assumption \ref{ass3low}.

\subsubsection{$\wdt{\beta}$ and $\wdt\lb$ satisfy the optimality conditions}
The proof is exactly the same as in Section \ref{oracleoptlow} after replacing (\ref{low}) by (\ref{lowC}).

\subsubsection{Conclusion of the proof}
The two preceding sub-sections prove that $\ch{\beta}$ has same support and sign pattern as
$\beta$. This occurs under the same conditions as those mentioned in the conclusion of the proof of Theorem \ref{main2}, Sec. \ref{concl-theo1}, plus the bound on 
$\| \gP_{V_T}z\|_2$ in Sec. \ref{VT}. 
Hence, this occurs with probability at least
\bean
1-\frac{216 + 10 + 2+1}{p^{\alpha}},
\eean
as announced.

\subsection{Epilogue: Nonempty range for $\|\beta\|_1$}

We need to ensure that the range of admissible values for $\beta$ is sufficiently large. The intuition 
says that this can be achieved by allowing sufficiently large values of $n$. In other 
words, we would like to know the additional constraints on the various parameters ensuring \bean
s\ \mathsf{L}_{\alpha,r,C}^{n,s,p}  & < & \mathsf{M}_{\alpha,r,\theta,C}^{n,s,p}.
\eean 
It then suffices to know when the following inequalities are satisfied:
\bea 
\label{firstCn}
m_{\alpha,r,C} \ s\ \frac{\sqrt{n-s}+\sqrt{2 \alpha \log p}}{\sqrt{s}} & \le & 
\frac{n-s}{\sqrt{\log p}} \left(\frac{\sqrt{\pi (n-s)}}{p^{\alpha}} \right)^{\frac{4}{n-s}}
\eea  
where 
\bean
m_{\alpha,r,C} & = & 6 \kappa \frac{\sqrt{1+2C}}{\sqrt{1-r}}.
\eean 
First, notice that under the condition 
\bea
\label{Cn}
n-s & \ge & 8\alpha \ s\ \log p \ge 8\alpha \log p,
\eea
we have $ \log  \left(\frac{\sqrt{\pi (n-s)}}{p^{\alpha}} \right)^{\frac{4}{n-s}}  =   \frac4{(n-s)}\Big( \frac12 \left(\log (\pi) +\log(n-s)\right)
-\alpha \log p\Big)  \ge - \frac12 $, and then
\bean
e^{-1/2} & \le & \left(\frac{\sqrt{\pi (n-s)}}{p^{\alpha}} \right)^{\frac{4}{n-s}}.
\eean

Therefore, since we also have $\sqrt{2 \alpha \log p}  \le \sqrt{n-s}$,  (\ref{firstCn}) is fulfilled if
\bean
2 m_{\alpha,r,C} \sqrt{s} \ \sqrt{n-s} & \le & e^{-1/2}\frac{n-s}{\sqrt{\log p}},
\eean
i.e.
\bean
n-s & \ge & 4e\ m_{\alpha,r,C}^2 \ s\ \log p.
\eean
This explains the constraint (\ref{n-main3}) with the constant $c_\circ:=4e\ m_{\alpha,r,C}^2>8\alpha$.

\section{Algorithms and simulations results}
\label{Simuls}
In this section, we propose one iterative algorithm for Strategies (A) and (B) and we study their 
practical performance via Monte Carlo experiments.

We performed Monte Carlo experiments in the following setting. We took $p=600$,
$n=75$ and $s=9$ and we ran 500 experiments with $\sigma^2=1$ and the coefficients 
of $\beta$ were randomly drawn independently as $B$ times a Bernoulli $\pm 1$ random variable
plus an independent centered Gaussian perturbation with variance one. 

\subsection{Preliminaries}
Our algorithms will be well defined under the assumption that for each positive value of the 
relaxation parameter, the value $\ch{\beta}_\lb$ of the regression vector is unique and the 
trajectory of $\ch{\beta}_\lb$ is continuous and piecewise affine. This 
property is well known under various assumptions on the design matrix $X$. It is a basic prerequisite 
for the theory behind Least Angle Regression and Homotopy methods. We refer the reader to \cite{Osborne:IMANA00} 
and \cite{Efron:AnnStat04} for information on these problems. 
See also \cite{Dossal:HAL11} for a recent account on the study of $\ch{\beta}_\lb$ as a function of $\lb$ under generic conditions on the design matrix. 

The subgradient conditions for the LASSO imply that 
\bea 
X_{\ch{T}_\lb}^t(y-X_{\ch{T}_\lb}\ch{\beta}_{\ch{T}_\lb}) & = & \lb \ \sgn(\ch{\beta}_{\ch{T}_\lb}).
\eea 
where $X_{\ch{T}_\lb}$ is non-singular, and we obtain the well known fact that,  
for any $\lambda>0$ such that $\ch{\beta}_\lambda\neq 0$, 
\bea
\label{turtle}
\ch{\beta}_{\ch{T}_{\lb}} & = & (X_{\ch{T}_\lb}^tX_{\ch{T}_\lb})^{-1}\left(X_{\ch{T}_\lb}^ty-\lambda \ \sgn\left(\ch{\beta}_{\ch{T}_\lb}\right)\right).
\eea

The following result is straightforward but useful. 
\begin{lemm}{\bf (Nontriviality of the estimator)}
\label{triv}
Let $\Sigma$ be the set 
\beq
\Sigma = \left\{ (S,\delta);\ S\subset \{ 1,\ldots,p\},\ 
\delta\in \{-1,1\}^{|S|},\ |S|\le n,\ \sigma_{\min}(X_{S})>0 \right\}.
\eeq
The inequality
\bea
\inf_{(S,\delta)\in\Sigma} \left\|(X_S^tX_S)^{-1} (X_S^ty-\lambda \delta)\right\|_{1} & > & 0
\eea 
holds with probability one.
\end{lemm}
\begin{proof}
This is an immediate consequence of the Gaussian distribution of $z$.
\end{proof}

\subsection{The standard LASSO with known variance}

\subsubsection{Simulations results: high SNR}
\label{simlasso1}

With the choice $B=40$, in all of the 500 experiments, we found that the support was exactly recovered. 

\subsubsection{Simulations results: low SNR}
\label{simlasso2}

Figure \ref{sig}
below shows the histogram of the number of properly recovered components (left column) and 
the number of false components (right column) for the LASSO estimator with known variance 
and $\lambda=2\sigma \sqrt{2\log p}$. 

\begin{figure}[htb]
\label{sig}
\begin{center}
\includegraphics[width=8cm]{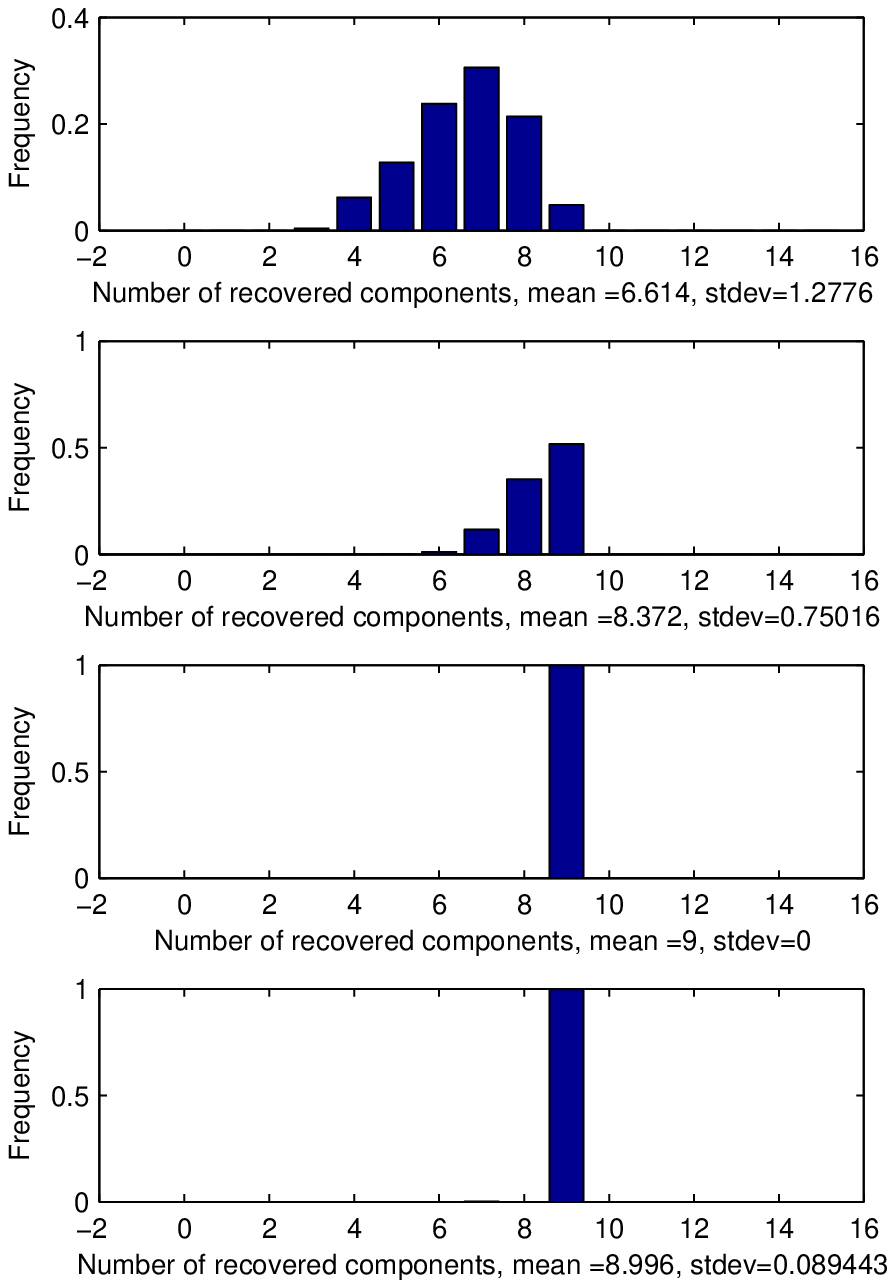}
\includegraphics[width=8cm]{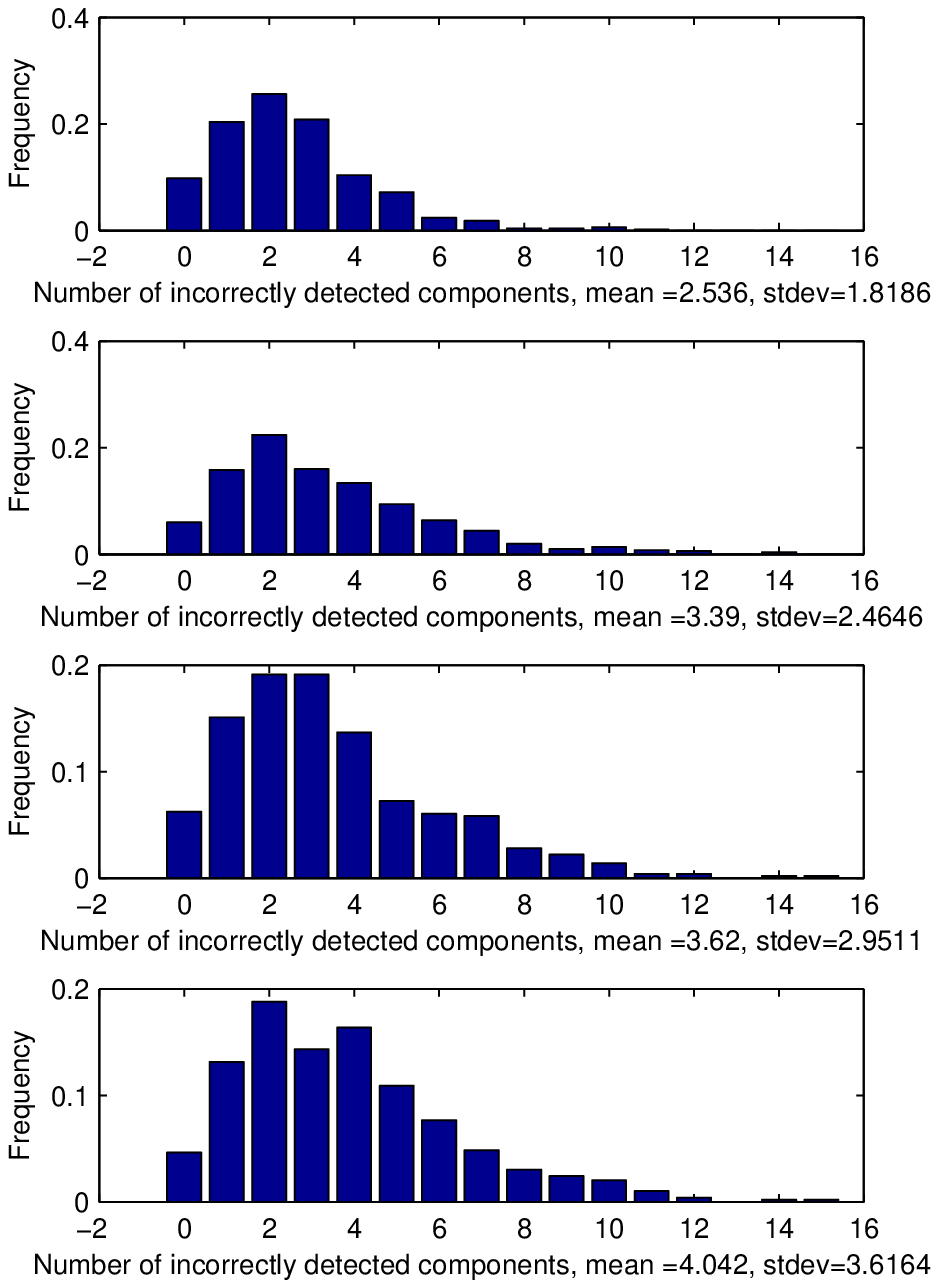}
\caption{Histogram of the number of properly recovered components (left column) and 
the number of false components (right column) for the LASSO estimator with known variance 
and $\lambda=2\sigma \sqrt{2\log p}$ for coeff. mean level $B=1,2,5,10$ (from top to bottom)}
\end{center}
\end{figure}

\subsection{Strategy (A)}

\subsubsection{The algorithm} \label{algoA}
As was discussed in Section \ref{defStratA}, finding the estimator $(\ch{\beta},\ch{\lb})$ in Strategy A is equivalent 
to solving the equation 
\bean 
\Gamma_A(\lb) & = & C_{{\rm var}}. 
\eean 
Since the function $\Gamma_A$ is increasing (see Appendix \ref{gammaA}), there is a number of Newton-type methods 
which can be used to solve this equation very efficiently and globally, i.e. without any condition on the 
initial iterate $\lb^{(0)}$; see e.g. \cite{Ralph:MOR94}. Instead of such a refined method, we propose
below a simpler fixed point iteration which was observed to work very well in practice. 

\vspace{.3cm}

\begin{algorithm}
\label{fixedpt}
\caption{Fixed point iterations for the LASSO with unknown variance}
\begin{algorithmic}
\STATE {\bf Input} $\lb^{(0)}$, $l=1$ and $\epsilon>0$         

\WHILE {$|\lb^{(l+1)}-\lb^{(l)}|\ge \epsilon$}
   \STATE Compute $\ch{\beta}_{\lb^{(l)}}$ as a solution of the LASSO problem
\bea
\label{lassl}
\ch{\beta}_{\lb^{(l)}} & \in & \down{b\in \mathbb R^p}{\rm argmin} \ \frac{1}{2} \|y-X b\|_2^2+ \lb^{(l)} \|b\|_1
\eea
   \STATE Set $\lb^{(l+1)}  =  \frac{C_{{\rm var}} }{\sqrt{n}}\|y-X \ch{\beta}_{\lb^{(l)}}^{(l)}\|_2 $      
   \STATE $l \leftarrow l+1$
\ENDWHILE
\STATE Set $\ch{\lb}^{(L)}=\lb^{(L)}$, $\ch{\beta}^{(L)}=\ch{\beta}_{\lb^{(L)}}$ and 
$\ch{\sigma}^{(L)}=\frac1{\sqrt{n}}\|y-X \ch{\beta}_{\lb^{(l)}}^{(l)}\|_2 $.
\STATE {\bf Output} $\ch{\beta}^{(L)}$, $\ch{\sg}^{(L)}$ and $\ch{\lb}^{(L)}$.
\end{algorithmic}
\end{algorithm}

Notice that the first step of the fixed point iteration procedure is similar to the correction of the 
standard estimator of $\sigma$ proposed by \cite{Sun:Test10}.

\subsubsection{Simulations results: high SNR}
\label{simsighat1}

As for the case of the standard LASSO with 
known variance of Section \ref{simlasso1} we found that, for $B=40$, the support was exactly recovered in all of 
the 500 experiments. 

\subsubsection{Simulations results: low SNR}
\label{simsighat2}
We performed Monte Carlo experiments in the same setting as for the LASSO in Section \ref{simlasso2}. 

In real situations where the level of magnitude of the regression coefficients may not be 
much higher than the noise level, one observes that false positives often occur for 
the LASSO estimator with known variance. 
As seen from these results, the LASSO estimator where the variance is estimated using the 
penalty $\ch{\lambda}=2\ch{\sigma} \sqrt{2\log p}$ performs at least as well as the standard 
LASSO estimator to which the true variance is available. The estimator $\ch{\sg}$ of the standard deviation 
is a slightly biased as shown in Figure \ref{histhatsig}.

\begin{figure}[htb]
\label{sighat}
\begin{center}
\includegraphics[width=8cm]{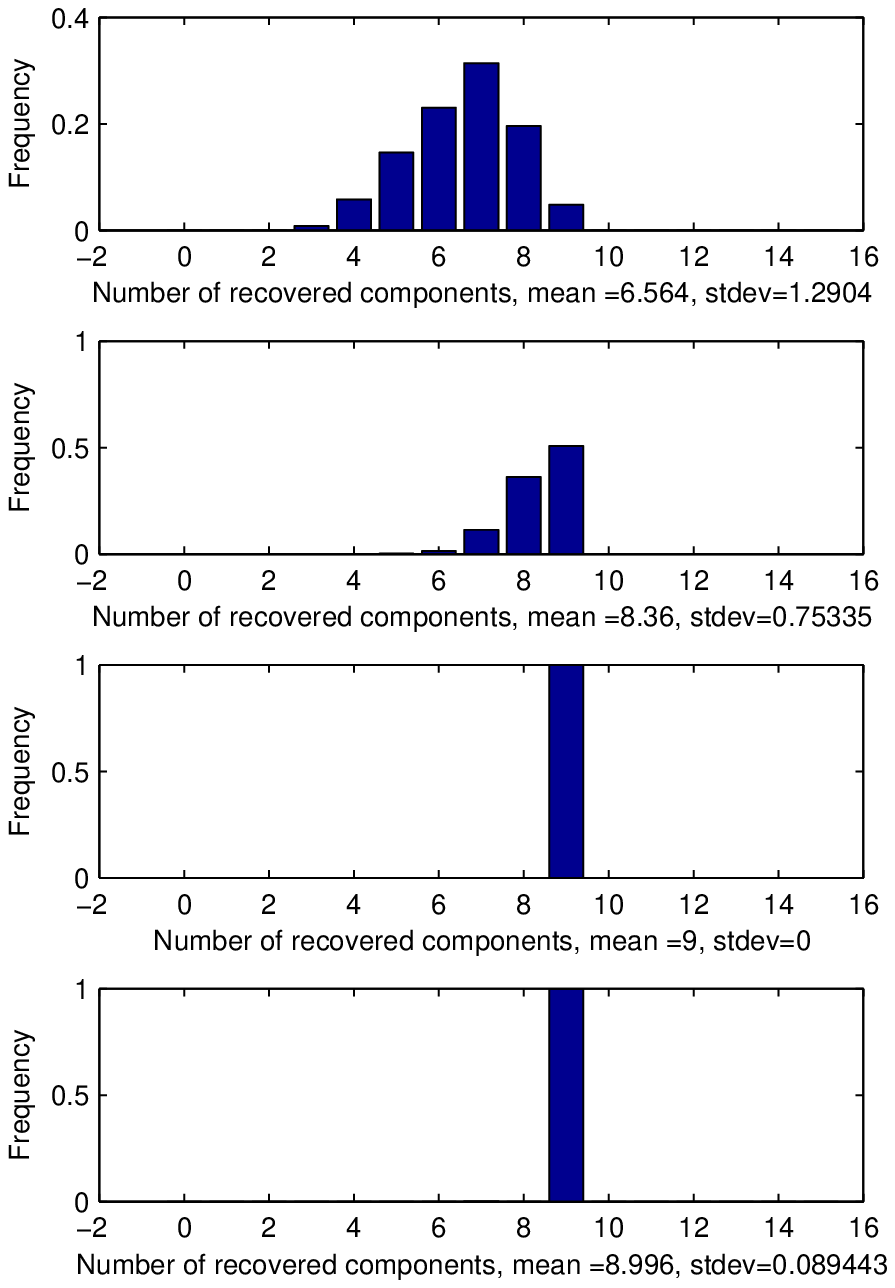}
\includegraphics[width=8cm]{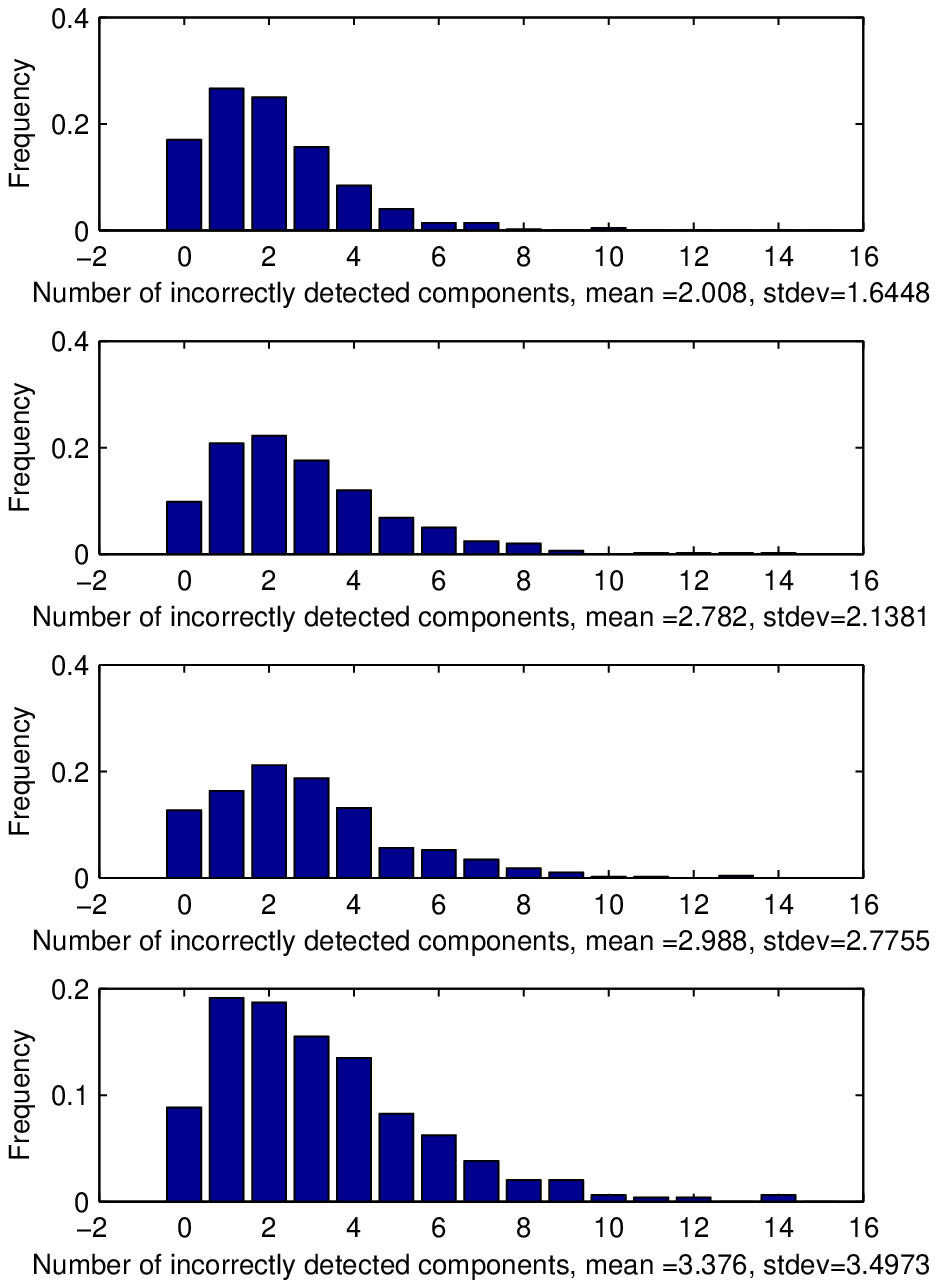}
\caption{Histogram of the number of properly recovered components (left column) and 
the number of false components (right column) for the LASSO estimator with unknown variance 
using Strategy (A) and $\ch{\lambda}=2\ch{\sigma} \sqrt{2\log p}$ for coeff. mean level $B=1,2,5,10$ (from top to bottom).}
\end{center}
\end{figure}

\begin{figure}[htb]
\label{histhatsig}
\begin{center}
\includegraphics[width=13cm]{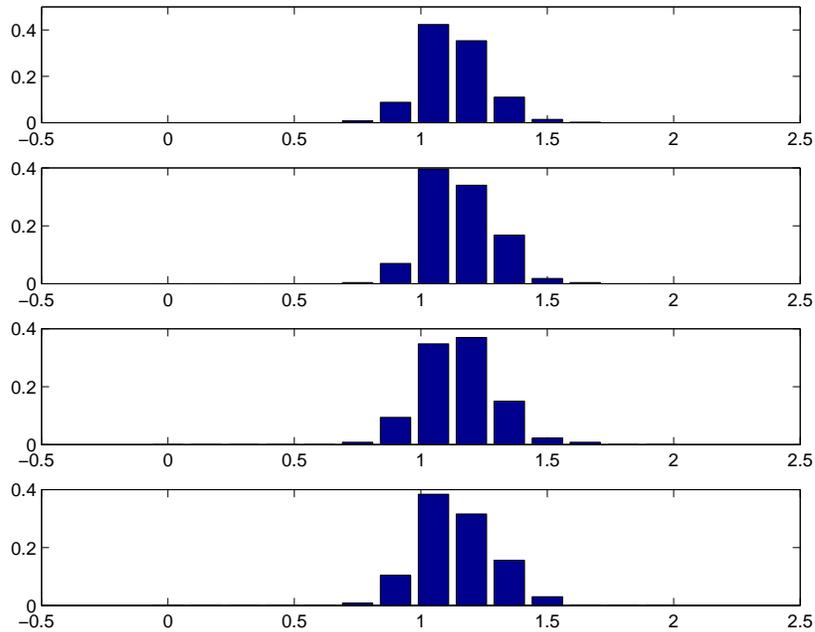}
\caption{Histogram of $\ch{\sg}$ for the LASSO estimator with unknown variance 
using Strategy (A) and $\ch{\lambda}=2\ch{\sigma} \sqrt{2\log p}$ for coeff. mean level $B=1,2,5,10$ (from top to bottom).}
\end{center}
\end{figure}

\subsection{Strategy (B)}

\subsubsection{The algorithm}
Basic computations show that the derivative of $\Gamma$ with respect to $\lambda$ is given by 
\bean
\frac{d \Gamma_B}{d\lambda}(\lambda)
& = & \frac{-\|\ch{\beta}_{\ch{T}_\lb}\|_1 - 
\lambda \ \|(X_{\ch{T}_\lb}^tX_{\ch{T}_\lb})^{-1/2} {\rm sign}\left(\ch{\beta}_{\ch{T}_\lb}\right)\|_2^2 
}
{\|y-X\ch{\beta}_\lb\|_2^2}
\eean
on each $\mathring{I}_k$, where $\mathring{I}_k$, is a maximal open interval on which the support of $\ch{\beta}_\lb$ 
is constant, for $k\in \mathcal K$ and $\cup_{k\in \mathcal K} I_k$ is a connected interval of 
$[0,+\infty)$. See for instance \cite[Section 4]{ChretienDarses:ArXiv11}.  

In order to compute the LASSO estimators $(\ch{\beta},\ch{\lambda})$ satisfying 
the penalty vs. fidelity tradeoff constraint, we need to find $\ch{\lambda}$ 
such that $\Gamma_B(\ch{\lambda})=C$. Since $\Gamma$ is strictly 
decreasing by Lemma \ref{gammafunc}, this task is not difficult to perform. 
A simple Newton-Raphson procedure for 
solving this equation is summarized in Algorithm \ref{nwt} below.

\vspace{.3cm}

\begin{algorithm}
\label{nwt}
\caption{Newton's method for the LASSO with penalty vs. fidelity tradeoff constraint}
\begin{algorithmic}
\STATE {\bf Input} $\lb^{(0)}$, $l=1$ and $\epsilon>0$         

\WHILE {$|\lb^{(l+1)}-\lb^{(l)}|\ge \epsilon$}
   \STATE Compute $\ch{\beta}_{\lb^{(l)}}$ as a solution of the LASSO problem
\bea
\label{lassl2}
\ch{\beta}_{\lb^{(l)}} & \in & \down{b\in \mathbb R^p}{\rm argmin} \ \frac{1}{2} \|y-X b\|_2^2+ \lb^{(l)} \|b\|_1
\eea
   \STATE Set $\lb^{(l+1)}  =  \lb^{(l)} - \displaystyle{\frac{d \Gamma_B}{d\lb}(\lb^{(l)})^{-1}}
\ \Gamma(\lb^{(l)})$        
   \STATE $l \leftarrow l+1$
\ENDWHILE
\STATE Set $\ch{\lb}^{(L)}=\lb^{(L)}$, $\ch{\beta}^{(L)}=\ch{\beta}_{\lb^{(L)}}$ and let $\ch{\sg}^{(L)^2}$ 
be given by 
\bea
\ch{\sigma}^{(L)^2} & = & \displaystyle{\frac{\|y-X \ch{\beta}^{(L)}\|_2^2 + 
2\ch{\lb}^{(L)} \ \|\ch{\beta}^{(L)}\|_1}{n}}
\eea
\STATE {\bf Output} $\ch{\beta}^{(L)}$, $\ch{\sg}^{(L)^2}$ and $\ch{\lb}^{(L)}$.
\end{algorithmic}
\end{algorithm}

\subsubsection{Simulations results: high SNR}
\label{simtradeoff1}

As for the case of the standard LASSO with 
known variance of Section \ref{simlasso1} we found that, for $B=40$, the support was exactly recovered in all of 
the 100 experiments. 

\subsubsection{Simulations results: low SNR}
\label{simtradeoff2}

We performed Monte Carlo experiments in the same setting as for the LASSO in Section \ref{simlasso2}. 

Figure \ref{tradeoff}
below shows the histogram of the number of properly recovered components (left column) and 
the number of false components (right column) for the LASSO estimator with unknown variance 
and the penalty vs. fidelity tradeoff constraint for the values $C=.01,.1,.5$. The 
instances where Newton's iterations did not converge were simply discarded although 
implementing a line search or a trust region strategy could easily have produced a 
correct result at the price of increasing the computational time for the Monte Carlo simulations
study.  
\begin{figure}[htb]
\label{tradeoff}
\begin{center}
\includegraphics[width=8cm]{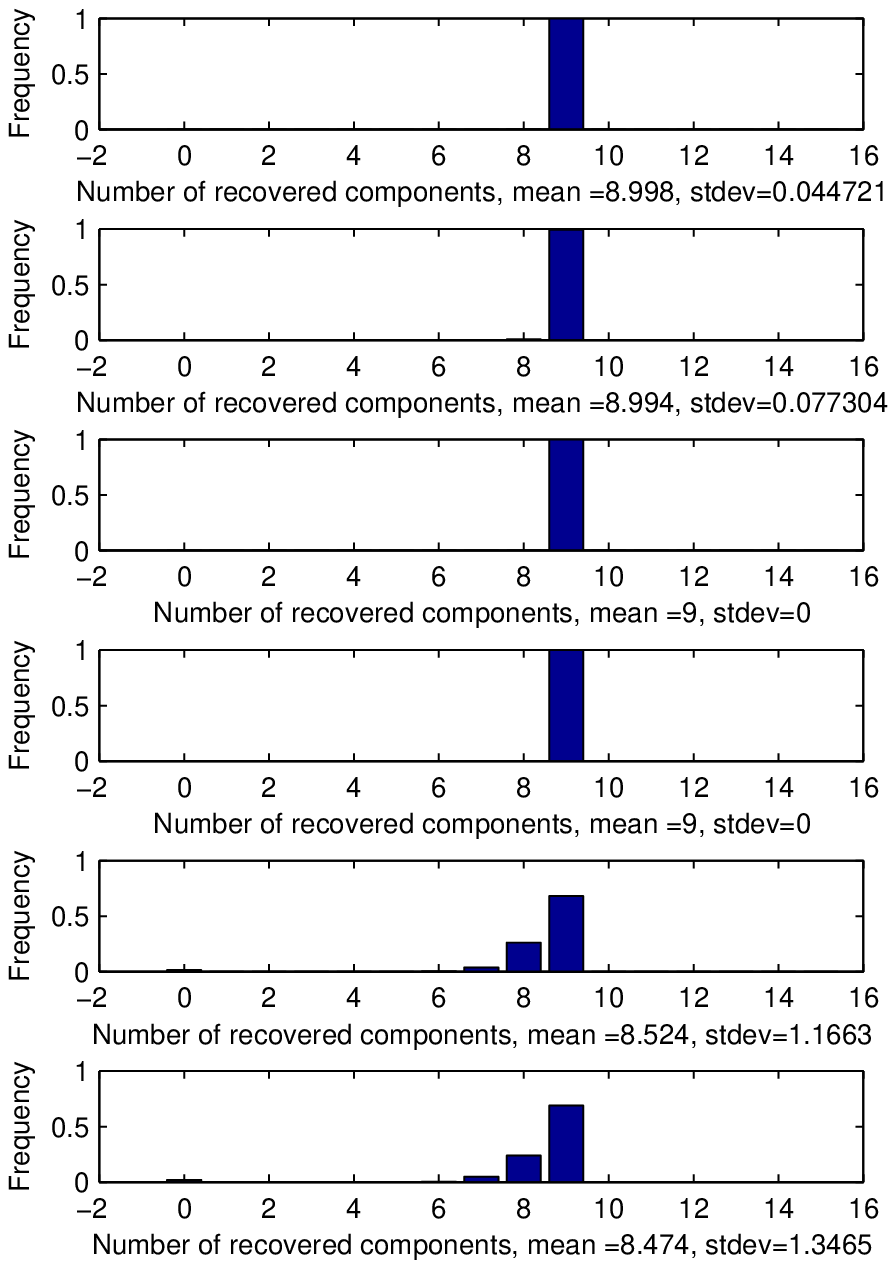}
\includegraphics[width=8cm]{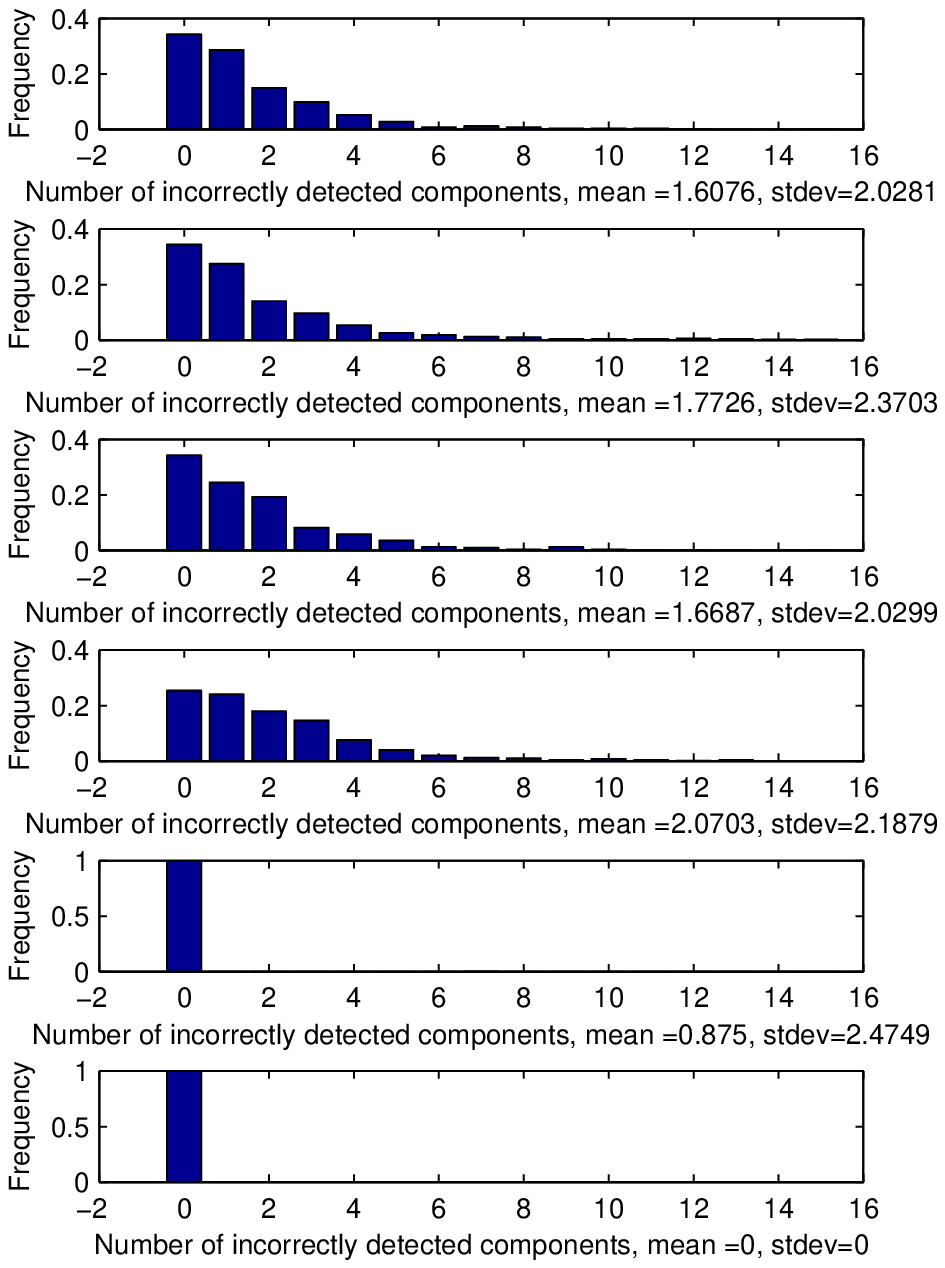}
\caption{Histogram of the number of properly recovered components (left column) and 
the number of false components (right column) for the LASSO estimator with unknown variance 
using Strategy (B) for $C=0.01,.01,0.5,1,5,10$ (from top to bottom) with level $B=5$.
}
\end{center}
\end{figure}

The number of well recovered components of $\beta$ is always equal to the 
true value 9 as $C$ increases for all values of $C$. On the other hand, the number of false positives 
increases with $C$. Our estimator with penalty vs. fidelity tradeoff constraint is seen to have quite better 
performances than the standard LASSO and LASSO with estimated variance of the previous section 
with respect to the number of false positives; 
compare Figure \ref{tradeoff} with the second row of Figure \ref{sig} or Figure \ref{sighat}. 
This was the main objective for proposing this strategy and the presented simulations show 
encouraging evidence of its robust behavior in the low SNR case. 
The low dependency on $C$ is a property which might be well appreciated 
in practice when neither the signal nor the noise levels are precisely known ahead of time.



\subsection{Comments}
The simulations results confirmed the theoretical findings that, in the high SNR case, 
Strategy (A) and Strategy (B) perform as well, without knowing the variance ahead of time, as the standard LASSO which 
uses the true value of the variance. Although the results are presented for a particular set of 
parameters, this behavior was observed more generally for a large number of numerical experiments with 
different parameter configurations, for which the standard LASSO exactly recovers the true support and sign 
pattern. In the low SNR setting, the standard LASSO and Strategy (A) perform poorly 
in the sense that many false components are selected. 
The Monte Carlo experiments show that Strategy (B) is more robust in the low SNR setting, in the sense that 
the estimated support contains much less false components. Surprisingly, this phenomenon was observed 
over a wide range of values for the constant $C$. In other words, the dependence of Strategy (B)'s 
performance on $C$ appeared as rather unessential for the recovery problem in the low SNR setting. 
As a preliminary practical conclusion, Strategy (A) appeared to be more suitable for the high SNR setting 
and Strategy (B) more suitable for the low SNR setting. In practice, the choice of $C$ in Strategy B could be based on standard model selection procedures (AIC, BIC, Foster and George, etc) for comparing the obtained supports over a large range of possible values. The limited number of possible supports occurring in practice as $C$ varies makes this comparison numerically tractable. 

Another interesting question is the one of the convergence of the algorithms proposed for Strategies (A) and (B). 
From the practical viewpoint, let us report that convergence was observed except in rare cases during the Monte Carlo 
experiments. However, we decided not to pursue their theoretical analysis here, since 
more robust methods enjoying global convergence have been proposed in the literature during the last 
twenty years. We refer the interested reader to e.g. \cite{Ralph:MOR94} for a globally convergent damped Newton method.  
Such methods are however more delicate to implement and the algorithms proposed in the present paper seem 
to be a good choice to start with in most practical experiments. 

Finally, there remains the question of choosing between Strategy A and Strategy B on a given practical problem. One reasonable way to proceed might simply be as follows: compare the supports obtained via both methods, using a standard model selection procedure such as BIC, AIC, Foster and George's criterion, etc.

\appendices


\section{Proof of Proposition \ref{linfs}} \label{cp-cond}

First, let us recall a technical result we obtained in \cite{cd_inv}:

\begin{lem}\label{norm12}
The following bound holds:
\bea
\bP\left( \|RH\|_{1\rightarrow 2}\ge v\right) &\le  & p  \left(e \frac{s}{p}\frac{\|X\|^2}{v^2} \right)^{v^2/\mu(X)^2},
\eea
provided that\ $e \frac{s}{p}\frac{\|X\|^2}{v^2}\le 1$.
\end{lem}

\label{linfsproof}
Let us introduce the events:
\bean
E & = & \{\|X_T^t X_T-\Id\| \le r\}\\
B & = & \left\{ \|RH\|_{1\rightarrow 2}\le \frac{c}{\sqrt{\log p}} \right\}.
\eean

The proofs that Conditions $(i)$ and $(ii)$ hold with high probability are trivial modifications of the ones 
given in \cite{CandesPlan:AnnStat09} up to the constants. The proofs that Conditions $(iii)$ and $(iv)$ hold with high probability can be performed using the following by-product inequality from our Lemma \ref{norm12}:
\bea \label{norm12bis}
\bP\left( B^c \right) &\le  & p \exp \left( \frac{c^2}{C_\mu^2} \log\left(e\frac{C_{\rm spar}}{c^2}\right) \ \log p  \right),
\eea
instead of using \cite[Lemma 3.5]{CandesPlan:AnnStat09} and \cite[Lemma 3.6]{CandesPlan:AnnStat09}. Here, we take
\bea \label{cc}
c^2 & \ge & \max(e^2C_{\rm spar} \ ; (1+\alpha)C_\mu),
\eea
so that 
\bea \label{norm12ter}
\bP\left( B^c \right) &\le  & \frac{1}{p^\alpha}.
\eea
All the proofs are moreover based on the simple inequality 
\bean
\bP (\mathcal  A) & = & \bP (\mathcal  A\cap E\cap B) + \bP (\mathcal  A\cap (E^c\cup B^c))\\
& \le &{\rm E} \left[\bP\left(\mathcal  A \mid R\right) \: \1_{E\cap B} \right] + \bP (E^c) + \bP (B^c),
\eean
and the bound, for a given vector $W$:
\bea \label{hoeff}
\bP\left( | \la W, X \ra| >t \right) & \le &  2  e^{-t^{2}/(2 \|W\|_2^2)}.
\eea
This last bound holds true for sub-Gaussian random vectors with independent components having Bernoulli or standard Gaussian 
distribution, for instance.

\subsection{Condition $(i)$}

Here, let $W_{i}$ be the $i$th row of  $(X_T^t X_T)^{-1}X_T^t$. Since $\la W_{i}, z \ra \sim \mathcal  N(0,\|W_i\|_2^2)$, we have from (\ref{hoeff}) and the union bound:
\bean
\bP\left( \max_{i\in T} | \la W_{i}, z \ra| >t \right) & \le &  2s \ e^{-t^{2}/(2 \max_i\|W_i\|_2^2)}.
\eean
Note that on $E$:
\beq
\max_{i\in T} \| W_{i} \|_{2} \ \le \ \|(X_T^t X_T)^{-1}\|\  \|X_T^t \| \ \le \frac{\sqrt{1+r}}{1-r}.
\eeq
One then obtains
\bean
\bP \left(\|(X_T^t X_T)^{-1}X_T^t z\|_{\infty} \le  \sigma\ \kappa \sqrt{\log p} \right) &\ge & 1-\frac{2}{p^{\alpha}},
\eean
whenever 
\bea \label{kappa1}
\kappa & \ge & \frac{\sqrt{2(1+\alpha)(1+r)}}{1-r}.
\eea

\subsection{Condition $(ii)$} Let us show that the estimate $(ii)$ holds with high probability. This is an actual consequence of our Lemma \ref{norm12}.

First, as in \cite{CandesPlan:AnnStat09} p.2171 and Lemma 3.3 p.2166, we write the inequality
\bean
\|(X_T^t X_T)^{-1}\sgn(\beta_T)\|_{\infty} & \le & 1+\max_{i\in T}|\la W_{i},\sgn(\beta_{T})\ra |,
\eean
where $W_{i}$ is the $i$th row or column of $(X_T^t X_T)^{-1}-\Id$. Set
\bean
\mathcal  A & =& \left\{ \max_{i\in J}|\la W_{i},\sgn(\beta_{T})\ra | \ge 2 \right\}.
\eean
Hoeffding's inequality yields:
\beq
\bP \left(\mathcal  A | R \right) \ \le \ 2|J| \exp\left(-\frac{2^{2}}{2\d \max_{i\in J}\| W_{i}\|_2^{2}}\right).
\eeq
As in \cite{CandesPlan:AnnStat09} p.2171 and p.2172, we write
$\|W_{i}\|_2  \le  \frac{\|RHR e_{i}\|_2}{1-r}$.
Thus on $E$:
\bean
\|W_{i}\|_2  \le  \frac{\|RHR\|_{1\rightarrow 2}}{1-r}\le  \frac{\|RH\|_{1\rightarrow 2}}{1-r}.
\eean
Recall that $\bP (B^c)\le \frac{1}{p^\alpha}$ since $c$ satisfies (\ref{cc}). Moreover
\bean
{\rm E} \left[\bP\left(\mathcal  A \mid R\right) \: \1_{E\cap B} \right] & \le &  \frac{1}{p^\alpha}
\eean
holds true if 
\bean
c^2 & \le & \frac{2(1-r)}{1+\alpha}.
\eean
We can easily check that this last condition is compatible with (\ref{cc}) and
\bean
C_\mu & = & \frac{r}{1+\alpha}\\
C_{\rm spar} & = & \frac{r^2}{(1+\alpha)e^2},
\eean
whenever $r\in(0,1/2)$. Therefore, when $r\in(0,1/2)$, the event 
$$\|(X_T^t X_T)^{-1}\sgn(\beta_T)\|_{\infty}  \le 1 +2 =3$$ 
holds with probability at least $1 -\frac{3}{p^{\alpha}}$.

\subsection{Condition $(iii)$}

Here, $W_{i}=(X_T^t X_T)^{-1}X_T^tX_i$. Notice that on $E\cap B$:
\beq
\max_{i\in T^c} \| W_{i} \|_{2} \  \le \frac{c}{(1-r) \sqrt{\log p}}.
\eeq

Using (\ref{norm12ter}) again and the same previous arguments, we obtain
\bean
\bP\left(\|X_{T^c}^t X_T (X_T^t X_T)^{-1}\sgn(\beta_T)\|_{\infty}\le \frac14\right) &\ge &1-\frac{3}{p^{\alpha}}.
\eean

\subsection{Condition $(iv)$}
If one now sets $W_{i}$ as the $i$th row of  $\Id -X_T (X_T^t X_T)^{-1}X_T^t $ and note that on $E$ for any $i\in T$:
\beq
\| W_{i} \|_{2}  \le  \|X_i\|_{2} =1,
\eeq
then:
\bean
\bP \left(\|X_{T^c}^t \left( I -X_T (X_T^t X_T)^{-1}X_T^t \right) z\|_{\infty}\le \sigma \ \kappa \ \sqrt{\log p}\right) & \ge & 1-\frac{2}{p^{\alpha}},
\eean
whenever
\bea \label{kappa2}
\kappa & \ge & \sqrt{2(1+\alpha)}.
\eea

\subsection{Choosing $\kappa$}
The parameter $\kappa$ has to satisfy (\ref{kappa1}) and (\ref{kappa2}). Since $r\in(0,\frac12]$, one has $\frac{\sqrt{2(1+r)}}{1-r}\le 2\sqrt{3}\approx 3.4$. Thus we simply chose 
\bean
\kappa & = & 4\sqrt{1+\alpha},
\eean 
which is Eq. (\ref{kappa}).

\section{Some properties of the $\chi^2$ distribution}

We recall the following useful
bounds for the $\chi^2(\nu)$ distribution of degree of freedom $\nu$
\begin{lemm} The following bounds hold:
\label{cki}
\bean
\bP\left(\chi(\nu) \ge  \sqrt{\nu}+ \sqrt{2t} \right) & \le & \exp(-t)\\
\bP\left(\chi(\nu) \le \sqrt{u \nu}\right)  & \le & \frac2{\sqrt{\pi \nu}}
\left(u\ e/2\right)^{\frac{\nu}4}.
\eean
\end{lemm}

\begin{proof}
For the first statement, see e.g. \cite{Massart:LNM??}. For the second statement, recall that 
\bean
\bP\left(\chi^2(\nu) \le u \nu \right)  & = & \int_0^{u\frac{\nu}{2}} \frac{t^{\frac{\nu}2-1}e^{-t}}{\Gamma(\frac{\nu}2)} dt\\
& = & \int_0^{u\frac{\nu}{2}} \frac{t^{\frac{\nu}2-1-\alpha}t^\alpha e^{-t}}{\Gamma(\frac{\nu}2)} dt.
\eean
Since $\max_{t\in \mathbb R^+} t^\alpha e^{-t}=(\alpha/e)^\alpha$ and is attained at $t=\alpha$, we obtain that 
\bean
\bP\left(\chi^2(\nu) \le u \nu \right) \ \le \ \frac{(\alpha/e)^\alpha}{\Gamma(\frac{\nu}2)} 
\int_0^{u\frac{\nu}{2}} t^{\frac{\nu}2-1-\alpha} dt \ = \ \frac{(\alpha/e)^\alpha}{(\frac{\nu}2-\alpha)\Gamma(\frac{\nu}2)} 
\left(u\frac{\nu}{2}\right)^{\frac{\nu}2-\alpha}.
\eean
Take for instance $\alpha=\frac{\nu}4$ and obtain 
\bea
\bP\left(\chi^2(\nu) \le u \nu \right) & = & \frac{(\nu/4e)^{\frac{\nu}4}}{\frac{\nu}4\Gamma(\frac{\nu}2)} 
\left(u\frac{\nu}{2}\right)^{\frac{\nu}4}.
\eea
On the other hand, we have 
\bean 
\Gamma(z) & \ge & \sqrt{2\pi}\ \frac{z^{z-\frac12}}{e^{z}}
\eean
and then, 
\bean 
\frac{(\nu/4e)^{\frac{\nu}4}}{\frac{\nu}4\Gamma(\frac{\nu}2)} & \le & \sqrt{\frac2{\pi}}
\frac{(e/2)^{\frac{\nu}4} \left(\frac{\nu}2\right)^{-\frac{\nu}4}}{\sqrt{\frac{\nu}2} }.
\eean 
Hence,
\bean
\bP\left(\chi^2(\nu) \le u \nu \right) & \le & \sqrt{\frac2{\pi}}
\frac{(u\ e/2)^{\frac{\nu}4} }{\sqrt{\frac{\nu}2} }  = \frac2{\sqrt{\pi \nu}}
\left(u\ e/2\right)^{\frac{\nu}4},
\eean
as desired.
\end{proof}

\section{Properties of the standard LASSO}
\label{lemgam}

\subsection{Reminders on the LASSO subgradient conditions}
\label{optlamb}
In \cite{Fuchs:IEEEIT04} Section III, it is proven that a necessary and sufficient optimality condition in (\ref{lass}) is the two following conditions:
\bea
\label{cond1}
X_T^t (y-X\mathsf{\ch{\beta}_\lb}) & = & \lb \: \sgn(\beta_T) \\
\label{cond2}
\|X_{T^c}^t (y-X\mathsf{\ch{\beta}_\lb})\|_{\infty} & \le & \lb.
\eea
Moreover, if $\|X_{T^c}^t (y-X\mathsf{\ch{\beta}_\lb})\|_{\infty} < \lb$, then problem (\ref{lass}) admits a unique solution.

Let us also recall (see \cite{Dossal:HAL11} and \cite{ChretienDarses:ArXiv11}) that the support $\ch{T}_{\lb}\subset \{1,\ldots,p\}$ of $\ch{\beta}_\lb$ verifies
\bea \label{support-lasso}
|\ch{T}_\lambda| & \le & n.
\eea

\subsection{General properties of $\lambda \mapsto \ch{\beta}_{\lambda}$}

Recall that $\ch{\beta}_{\lambda}$ is the standard LASSO estimator of $\beta$ parametrized by $\lambda$,

The following notations will be useful. Define $\mathcal  L$ as the cost function:
\bef{\mathcal  L}
(0,+\infty)\times \R^{p} &  \longrightarrow & \R_{+}\\
(\lb,b) & \longmapsto & \d \frac12 \|y-X b\|_2^2+\lb \|b\|_1,
\label{calL}
\eef
and for all $\lb>0$,
\bean
\theta(\lambda) & = & \inf_{b\in \mathbb R^p} \mathcal  L(\lb,b).
\eean

\begin{lemm}
\label{conc}
Let the Generic Condition hold. Then, the function $\theta$ is concave and non-decreasing.\\
\end{lemm}

\begin{proof}
Since $\theta$ is the infimum of a set of affine functions of the variable $\lambda$, it is concave.
Moreover, we have
\bean
\theta(\lambda) & = & \mathcal  L(\lb, \ch{\beta}_\lambda),
\eean
where, by Proposition \ref{propun}, $\ch{\beta}$ is the unique solution of (\ref{lasso}).
Using the filling property \cite[Chapter XII]{Hiriart:CAMA93}, we obtain that 
$\partial \theta(\lb)$ is the singleton $\{\|\ch{\beta}_{\lambda}\|_1\}$. 
Thus, $\theta$ is differentiable and its derivative at $\lambda$ is given by
\bean
\theta^\prime(\lambda) & = & \|\ch{\beta}_{\lambda}\|_1.
\eean
Moreover, this last expression shows that $\theta$ is nondecreasing.
\end{proof}

\subsubsection{Proof of Lemma \ref{cont}}  \label{lem-continuity}
\begin{itemize}
\item[(i)] {\it $\|\ch{\beta}_\lb\|_1$ is non-increasing} -- The fact that $\lambda \longmapsto \|\ch{\beta}_{\lambda}\|_1$ is non-increasing is an immediate consequence of the concavity of $\theta$.

\item[(ii)] {\it Boundedness} -- Notice that using (\ref{turtle}), we obtain that 
\bean
\| \ch{\beta}_{\lb}\|_{1} & \le & \max_{(S,\delta)\in\Sigma} \left\| (X_S^tX_{S})^{-1} (X_S^ty-\lb \delta )\right\|_{1}.
\eean
Thus, $\lambda  \longmapsto \ch{\beta}_{\lambda}$ is bounded on any interval of the form $(0,M]$, with 
$M\in (0,+\infty)$. Moreover, since its $\ell_{1}$-norm is non-increasing, it is bounded on $(0,\infty)$.
 \item[(iii)] {\it Continuity} --  
Assume for contradiction that $\lambda \longmapsto \ch{\beta}_{\lb}$
is not continuous at some $\lambda^{\circ}>0$. Using boundedness, we can construct two 
sequences converging towards 
$\ch{\beta}^{+}_{\lambda^{\circ}}$ and  $\ch{\beta}^{-}_{\lambda^{\circ}}$ respectively with 
$\ch{\beta}^{+}_{\lambda^{\circ}}\neq \ch{\beta}^{-}_{\lambda^{\circ}}$.
Since $\mathcal  L(\lb^{\circ},\cdot)$ is continuous, both limits are optimal solutions of the problem
\bea
\label{betlambsharp}
 \down{b\in \R^p}{\rm argmin} \ \mathcal  L(\lb^{\circ},b),
\eea
hence contradicting the uniqueness.
\end{itemize}

\subsubsection{Partitioning $(0,+\infty)$ into good intervals}

The continuity of $\lb \mapsto \ch{\beta}_\lb$ implies that the interval
$(0,+\infty)$ can be partitioned into subintervals of the type $I_k=(\lambda_k,\lambda_{k+1}]$,
with 
\begin{itemize}
\item[(i)]  $\lambda_0=0$ and $\lambda_{k}\in (0,+\infty]$ for $k> 0$,
\item[(ii)]  the support and sign pattern of $\ch{\beta}_{\lambda}$ are constant on each 
$\mathring{I}_k$. 
\end{itemize}
Notice further that due to {\it Step 1.a}, $\ch{T}_\lambda \neq \emptyset$ on at least $I_0$. Let
$\mathcal K$ be the nonempty set  
\bean
\mathcal  K & = & \left\{ k\in \N, \ \forall \lb \in \mathring{I}_k, \ \ch{\beta}_\lambda \neq 0 \right\}.
\eean
On any interval $I_k$, $k\in \mathcal K$, uniqueness of $\ch{\beta}$ implies that the expression (\ref{turtle}) 
for $\ch{\beta}_{\ch{T}_\lb}$ holds. Multiplying
(\ref{turtle}) on the left by $\sgn\left(\ch{\beta}_{\ch{T}_\lb}\right)^t$, we obtain
\bean
\|\ch{\beta}_\lambda\|_1  =  \sgn\left(\ch{\beta}_{\ch{T}_\lb}\right)^t(X_{\ch{T}_\lb}^tX_{\ch{T}_\lb})^{-1} X_T^ty  - \lambda \sgn\left(\ch{\beta}_{\ch{T}_\lb}\right)^t(X_{\ch{T}_\lb}^tX_{\ch{T}_\lb})^{-1} \sgn\left(\ch{\beta}_{\ch{T}_\lb}\right).
\eean
Thus
\bean
\frac{d \|\ch{\beta}_\lambda\|_1}{d\lambda}(\lambda)
& = &- \sgn\left(\ch{\beta}_{\ch{T}_\lb}\right)^t(X_{\ch{T}_\lb}^tX_{\ch{T}_\lb})^{-1} \sgn\left(\ch{\beta}_{\ch{T}_\lb}\right),
\eean
on $(0,+\infty)$. Thus, the definition of $\Sigma$, we obtain that
\bea
\label{slope}
\frac{d \|\ch{\beta}_\lambda\|_1}{d\lambda}(\lambda)
& \le & 
-\inf_{(S,\delta)\in \Sigma} \delta^t(X_{\ch{T}_\lb}^tX_{\ch{T}_\lb})^{-2} \delta \ <\ 0 
\eea
on each $\mathring{I}_k$, $k\in \mathcal K$ and
\bean
\frac{d \|\ch{\beta}_\lambda\|_1}{d\lambda}(\lambda) & = & 0
\eean
on each $\mathring{I}_k$, $k\not \in \mathcal K$, i.e. on each $\mathring{I}_k$ such that $\|\ch{\beta}_{\ch{T}_\lb}\|_1=0$
for all $\lb$ in $I_k$, if any such $I_k$ exists.
Since $\lambda \longmapsto \|\ch{\beta}_{\lambda}\|_1$ is continuous on $(0,\infty)$, (\ref{slope})
implies that:
\begin{itemize} 
\item[(i)]  there exists $\tau \in (0,+\infty)$, such that $\ch\beta_{\tau}=0$ (as an easy consequence of 
the Fundamental Theorem of Calculus and a contradiction).
\item[(ii)]  $\ch\beta_{\lb}=0$ for all $\lb\ge \tau$. 
\end{itemize}
Hence $\cup_{k\in \mathcal K}I_k$ is a connected bounded interval. 

\subsubsection{$|\ch{T}_{\lb}|=n$ for $\lb$ sufficiently small.} 
Let $(\lb_{k})_{k\in\N}$ be any positive sequence converging to $0$.
Let $\beta^{*}$ be any cluster point of the sequence $(\ch \beta_{\lb_{k}})_{k\in \N}$ 
(this sequence is easily seen to be bounded under various standard assumptions; 
see e.g. \cite[Lemma 3.5]{ChretienDarses:ArXiv11} for a proof). 
Fix $\e>0$ and $b\in\R^{p}$. For all $k\in\N$, we have
\bea
\mathcal  L (\lb_{k},\ch\beta_{\lb_{k}}) & \le & \mathcal  L(\lb_{k},b),
\eea
where $\mathcal L$ is defined by (\ref{calL}). 
Since $\mathcal  L(\lb_{k},\cdot)$ is continuous, we can also write for $k$ sufficiently large:
\bean
\mathcal  L (\lb_{k},\beta^{*}) &\le& \mathcal  L (\lb_{k},\ch\beta_{\lb_{k}}) + \e.
\eean
Hence,
$\mathcal  L (\lb_{k},\beta^{*}) \le \mathcal  L (\lb_{k},b) + \e$.
Letting $\lb_{k}\to 0$, we obtain
\bean
\frac12\|y-X \beta^{*}\|_2^2 &\le& \frac12\|y-X b\|_2^2 + \e,
\eean
and thus, 
\bea
\label{tempere}
\frac12\|y-X \beta^{*}\|_2^2 &\le& \inf_{b\in\R^{p}} \frac12\|y-X b\|_2^2.
\eea
Since ${\rm range} (X)=\R^{n}$, (\ref{tempere}) implies 
$\|y-X \beta^{*}\|_2^2 = 0$,
and then
\bea \label{limlb}
\lim_{\lb\downarrow 0}\|y-X \ch{\beta}_{\lambda}\|_2^2 & = & 0.
\eea
Notice further that $\{b\in\R^{p}, \ |{\rm supp}(b)|<n\}$ is a finite union of subspaces of 
$\R^{p}$, each with dimension $n-1$. Thus,
\bean
m:= \inf_{ \{b\in\R^{p}; \ |{\rm supp}(b)|<n\} } \frac12\|y-X b\|_2^2 &>& 0,
\eean
with probability one. Therefore for $\lb$ sufficiently small, (\ref{limlb}) implies
$\|y-X \ch{\beta}_{\lambda}\|_2^2  <  m$,
from which we deduce that $|\ch{T}_{\lb}|=n$ since one has $|\ch{T}_\lambda| \le  n$ (cf Reminder \ref{optlamb}).

\subsubsection{The map $\lb \mapsto\|y-X\ch{\beta}_\lb\|_2$ is increasing on $(0,\tau]$}

Using (\ref{turtle}), we obtain 
\bean
y-X \ch{\beta}_{\lambda}  & = & 
P_{V_{\ch{T}_\lb}^\perp}(y)
- \lb X_{\ch{T}_\lb} (X_{\ch{T}_\lb}^tX_{\ch{T}_\lb})^{-1} \sgn\left(\ch{\beta}_{\ch{T}_\lambda}\right),
\eean
which implies that 
\bean
\|y-X \ch{\beta}_{\lambda}\|_2^2 
& = & \left\| P_{V_{\ch{T}_\lb}^\perp}(y) \right\|_2^2 
-2 \lb \la P_{V_{\ch{T}_\lb}^\perp}(y),X_{\ch{T}_\lb} (X_{\ch{T}_\lb}^tX_{\ch{T}_\lb})^{-1} \sgn\left(\ch{\beta}_{\ch{T}_\lambda}\right) \ra \nonumber \\  
& & \hspace{2cm}+\lb^{2} \ \sgn\left(\ch{\beta}_{\ch{T}_\lambda}\right)^t 
(X_{\ch{T}_\lb}^tX_{\ch{T}_\lb})^{-1} \sgn\left(\ch{\beta}_{\ch{T}_\lambda}\right)
\eean
and thus, by the definition of $P_{V_{\ch{T}_\lb}^\perp}(y)$,
\bea
\label{denom}
\|y-X \ch{\beta}_{\lambda}\|_2^2 
& = & \left\| P_{V_{\ch{T}_\lb}^\perp}(y) \right\|_2^2 
+\lb^{2} \ \sgn\left(\ch{\beta}_{\ch{T}_\lambda}\right)^t 
(X_{\ch{T}_\lb}^tX_{\ch{T}_\lb})^{-1} \sgn\left(\ch{\beta}_{\ch{T}_\lambda}\right).
\eea

From (\ref{denom}), since $(X_{\ch{T}_\lb}^tX_{\ch{T}_\lb})^{-1}$ is definite, we obtain 
that $\lb \mapsto \|y-X\ch{\beta}_\lb\|_2$ is increasing on each $\mathring{I}_k$, and thus on $(0,\tau]$ by using that $\lb \mapsto \|y-X\ch{\beta}_\lb\|_2$ is continuous on $(0,\tau]$.

\subsection{Study of $\Gamma_A$} \label{sec-gammaA}

\begin{lemm} \label{gammaA}
$\Gamma_A$ is increasing on $(0,\tau]$ and $\lim_{\lb\rightarrow +\infty} \Gamma_A(\lb)=+\infty$.
\end{lemm}
\
\begin{proof}
Due to Step 3, and the definition of $\tau$, the set of values $\lb>0$ such that 
$\|y-X\ch{\beta}_\lb\|_2>0$ is nonempty. Let $\lb_{\rm inf}$ denote its infimum value.  
Take $\lb \in \mathring{I}_k$ for some $k$ such that $\lb\ge \lb_{\rm inf}$. In particular, 
$\lb \neq 0$. Then, 
\bea
\Gamma_A(\lb) & = & \frac{s}{\frac1{\lb^2} \left\| P_{V_{\ch{T}_\lb}^\perp}(y) \right\|_2^2 
+\sgn\left(\ch{\beta}_{\ch{T}_\lambda}\right)^t 
(X_{\ch{T}_\lb}^tX_{\ch{T}_\lb})^{-1} \sgn\left(\ch{\beta}_{\ch{T}_\lambda}\right)},
\eea 
and we deduce that $\Gamma_A$ is increasing on $\mathring{I}_k$. By continuity, we have that 
$\Gamma_A$ is increasing on $(\lb_{\rm inf},\tau]$. 
Once $\lb>\tau$, $\|y-X\ch{\beta}_\lb\|_2^2=\|y\|_2^2$ and $\Gamma_A(\lb)=s \lb^2/\|y\|_2^2$. Thus, 
$\lim_{\lb\rightarrow +\infty} \Gamma_A(\lb)=+\infty$ as desired. 
\end{proof}
\

The fact that $\Gamma_A$ is increasing 
proves that the equation $\Gamma_A(\lb)=C_{\rm var}$ admits at most one solution.

\subsection{Study of $\Gamma_B$}

Recall that 
\bea
\label{gamma}
\Gamma_B(\lb) & = & 
\frac{\lambda\|\ch{\beta}_{\lambda}\|_1}{\|y-X \ch{\beta}_{\lambda}\|_2^2}.
\eea
We will use repeatedly that $\ch{\beta}_\lb$ is unique for all $\lb>0$ and 
that the trajectory $\lb \mapsto \ch{\beta}_\lb$ is continuous under the Generic Condition, see \cite{Dossal:HAL11}.\\

\begin{lem} \label{gammafunc}
Under the Generic Position Assumption of \cite{Dossal:HAL11}, the function $\Gamma_B$ defined by (\ref{gamma}) almost surely satisfies
\bea
\lim_{\lambda\downarrow 0} \ \Gamma_B(\lambda) & = & +\infty.
\eea
Moreover, almost surely, there exists $\tau>0$ such that $\Gamma_B$ is decreasing
on the interval $(0,\tau]$ with $\Gamma_B(\tau)=0$, while $\|y-X\ch{\beta}_\lb\|_2$ is increasing on $(0,\tau]$.\\
\end{lem}

\begin{proof}
Let us first show that $\lim_{\lambda\downarrow 0}\Gamma_B(\lambda)= +\infty$.

Let $\lb_{0}>0$ be sufficiently small so that for all $\lb\le \lb_{0}$, $|\ch{T}_{\lb}|=n$. Such 
a $\lb_0$ exists due to {\it Step 1.a}. 
Hence, since $X_{\ch{T}_{\lb}}$ is nonsingular:
\bea
\gP_{V_{T_{\lb}}}  & = & \Id_n.
\eea
Thus, using (\ref{turtle}), we obtain 
\bea
y-X \ch{\beta}_{\lambda}  & = & - \lb X_{\ch{T}_\lb} (X_{\ch{T}_\lb}^tX_{\ch{T}_\lb})^{-1} \sgn\left(\ch{\beta}_{\ch{T}_\lambda}\right),
\eea
which implies that 
\bean
\|y-X \ch{\beta}_{\lambda}\|_2^2 & = &
\lb^{2} \|(X_{\ch{T}_\lb}^tX_{\ch{T}_\lb})^{-1} \sgn\left(\ch{\beta}_{\ch{T}_\lambda}\right)\|_{2}^{2}.
\eean
Moreover, Lemma \ref{triv} combined with (\ref{turtle}) gives
\bean
\|\ch{\beta}_{\lambda}\|_1 & > \inf_{(S,\delta)\in\Sigma} \left\|(X_S^tX_S)^{-1} (X_S^ty-\lambda \delta)\right\|_{1} := m'  
> & 0.
\eean
Hence, for $\lb\le\lb_{0}$,
\bean
\Gamma_B(\lb) & \ge & \frac{\lb m'}{\lb^{2} \|X_{\ch{T}_\lb} (X_{\ch{T}_\lb}^tX_{\ch{T}_\lb})^{-1} \sgn\left(\ch{\beta}_{\ch{T}_\lambda}\right)\|_{2}^{2}}.
\eean
Using the trivial fact that
$\sup_{(S,\delta)\in\Sigma}\|X_{S}(X_S^tX_S)^{-1} \delta\|_{2}^{2}  <  \infty$,
the proof of Step 1 is complete.\\

Let us now show that  $\Gamma_B$ is decreasing on $(0,\tau)$ by studying the function
\bef{\Phi}
(0,+\infty) & \longrightarrow & \R_{+} \\
\lambda & \longmapsto & \lambda \|\ch{\beta}_\lambda\|_1.
\eef
We immediately deduce from Step 2 and the definition of the intervals $I_k$, 
$k\in \mathcal K$, that $\Phi$ is differentiable on each $\mathring{I}_k$, 
$k \in \mathcal K$. Using (\ref{turtle}), its derivative on $\mathring{I}_k$ reads 
\bean
\frac{d \Phi}{d\lambda}(\lambda) & = &
\|\ch{\beta}_{\ch{T}_\lb}\|_1 -
\lambda \ \sgn\left(\ch{\beta}_{\ch{T}_\lb}\right)^t(X_{\ch{T}_\lb}^tX_{\ch{T}_\lb})^{-1} \sgn\left(\ch{\beta}_{\ch{T}_\lb}\right) \nonumber\\
& = & \label{prtt}  \|\ch{\beta}_{\ch{T}_\lb}\|_1 - 
\lambda \ \|(X_{\ch{T}_\lb}^tX_{\ch{T}_\lb})^{-1/2} \sgn\left(\ch{\beta}_{\ch{T}_\lb}\right)\|_2^2.
\eean
Now, since $X_{\ch{T}_\lb}$ is non singular, 
\bean 
\|y-X\ch{\beta}_\lb\|_2^2  =  \lambda^2 \|(X_{\ch{T}_\lb}^tX_{\ch{T}_\lb})^{-1} \sgn\left(\ch{\beta}_{\ch{T}_\lambda}\right)\|_{2}^{2}  >  \lambda^2 n \ \sigma_{\min}\left((X_{\ch{T}_\lb}^tX_{\ch{T}_\lb})^{-1}\right)^2 >  0
\eean 
for $\lb>0$. Therefore $\Gamma_B(\lambda)<+\infty$ on $(0,+\infty)$, $\Gamma_B$ is continuous on $I_k$ and differentiable 
on $\mathring{I}_k$. Moreover, using (\ref{denom}), we have
\bean
\frac{d \Gamma_B}{d\lambda}(\lambda) =  \frac{\frac{d \Phi}{d\lambda}(\lambda)\|y-X\ch{\beta}_\lb\|_2^2-\Phi(\lambda) \frac{d \|y-X\ch{\beta}_\lb\|_2^2}{d\lambda}(\lambda)}{\|y-X\ch{\beta}_\lb\|_2^4} = \frac{\frac{d \Phi}{d\lambda}(\lambda)- 2 \frac{\Phi(\lambda)}{\lb}}{\|y-X\ch{\beta}_\lb\|_2^2}.
\eean
Hence, using (\ref{prtt}) and (\ref{denom}),
\bean
\frac{d \Gamma_B}{d\lambda}(\lambda)
& = & \frac{-\|\ch{\beta}_{\ch{T}_\lb}\|_1 - 
\lambda \ \|(X_{\ch{T}_\lb}^tX_{\ch{T}_\lb})^{-1/2} \sgn\left(\ch{\beta}_{\ch{T}_\lb}\right)\|_2^2 
}
{\|y-X\ch{\beta}_\lb\|_2^2} \\
& \le & \frac{- \lambda \ \|(X_{\ch{T}_\lb}^tX_{\ch{T}_\lb})^{-1/2} \sgn\left(\ch{\beta}_{\ch{T}_\lb}\right)\|_2^2 
}
{\lb^{2} \|(X_{\ch{T}_\lb}^tX_{\ch{T}_\lb})^{-1} \sgn\left(\ch{\beta}_{\ch{T}_\lambda}\right)\|_{2}^{2}}
 \le  -\frac1{\lb} \left(\frac{\sigma_{\min}\left((X_{\ch{T}_\lb}^tX_{\ch{T}_\lb})^{-1/2}\right)}
{\sigma_{\max}\left((X_{\ch{T}_\lb}^tX_{\ch{T}_\lb})^{-1}\right)}\right)^2 ,
\eean 
on each $\mathring{I}_k$.  We can thus conclude, due to the non-singularity of $X_{\ch{T}_\lb}$, that 
$\Gamma_B$ is decreasing on $(0,\tau)$, as announced.

\end{proof}

\section*{Acknowledgment} The authors are very grateful to the referees and the editor for their thorough reading and helpful comments that yield to 
substantial clarifications and simplifications of the presentation and the arguments.

\bibliographystyle{amsplain}
\bibliography{database}

\begin{thebibliography}{10}


\bibitem{Baraud:AnnStat09} Baraud, Y., Giraud, C., Huet, S. Gaussian model selection with an unknown variance. Ann. Statist. 37 (2009),  no. 2, 630--672.

\bibitem{Bickel:AnnStat09}
Bickel, P. J., Ritov, Y., Tsybakov, A. B., Simultaneous analysis of lasso and Dantzig selector. Ann. Statist. 37 (2009), no. 4, 1705--1732.


\bibitem{Bourgain:IsrJM87} Bourgain, J., Tzafriri, L., Invertibility of ``large'' submatrices with applications to the
 geometry of Banach spaces and harmonic analysis. Israel J. Math. 57 (1987), no. 2, 137--224.

\bibitem{bsv} St\"adler, N., B\"uhlmann, P., and van de Geer, S. (2010), $\ell_1$-penalization for mixture regression models, Test, 19, 209--285

\bibitem{Bunea:EJStat07} Bunea, F., Tsybakov, A., and Wegkamp, M. (2007a). Sparsity oracle inequalities for the Lasso. Electron. J. Stat., 1 :169--194.

\bibitem{Candes:CRAS08} Cand\`es, E. J.  The restricted isometry property and its implications for compressed sensing. C. R. Math. Acad. Sci. Paris  346  (2008),  no. 9-10, 589--592.

\bibitem{Candes:ActaNum06} Cand\`es, E. J.  Modern statistical estimation via oracle inequalities. Acta Numer. 15 (2006), 257--325.


\bibitem{CandesPlan:AnnStat09}
Cand\`es, E. J. and Plan, Y. Near-ideal model selection by $\ell_1$ minimization.
 Ann. Statist. 37 (2009),  no. 5A, 285--2177.

\bibitem{Candes:InvProb07} Cand\`es, E. and Romberg, J., Sparsity and incoherence in compressive sampling. Inverse Problems 23 (2007), no. 3, 969--985.

\bibitem{CandesTao:AnnStat07} Cand\`es, E. J. and Tao, T., The Dantzig Selector: statistical estimation when $p$ is much larger than $n$. Ann. Stat.
35, no. 6 (2007), 2313--2351.

\bibitem{cd_inv} Chr\'etien, S. and Darses, S., Invertibility of random submatrices via tail decoupling and a Matrix Chernoff Inequality. Statist. Probab. Lett. 82 (2012), no. 7, 1479-1487.

\bibitem{ChretienDarses:ArXiv11} Chr\'etien, S. and Darses, S., The LASSO for generic design matrices as a function of the relaxation parameter, http://arxiv.org/abs/1105.1430.

\bibitem{Donoho:IEEEIT01} Donoho, D.L. and Huo, X., Uncertainty principles and ideal atomic decomposition. IEEE Trans. Inform.
Theory, 47 (2001) 2845--2862.

\bibitem{Donoho:PNAS03} Donoho, D.L. and Elad, M.. Optimally sparse representation in general (non-orthogonal) dictionaries
via 1 minimization. Proc. Natl. Acad. Sci. USA, 100 (2003) 2197--2202.

\bibitem{Dossal:HAL11} Dossal, C., A necessary and sufficient 
condition for exact recovery by $\ell_1$ minimization. \begin{verbatim}http://hal.archives-ouvertes.fr/hal-00164738/en/\end{verbatim}.

\bibitem{Efron:AnnStat04} Efron, B., Hastie, T., Johnstone, I. and Tibshirani, R. Least angle regression, Annals of Statistics, 32 (2004) 407--451.


\bibitem{Elad:IEEEIT02} Elad, M. and Bruckstein, A.M., A generalized uncertainty principle and sparse representation in pairs
of RN bases. IEEE Trans. Inform. Theory, 48 (2002) 2558--2567.

\bibitem{Fuchs:IEEEIT04} Fuchs, J.J., On sparse representations in arbitrary redundant bases, IEEE Trans. Inform. Theory, 
50 (2004) no. 6 1341--1344.


\bibitem{Gine:Decoupling99} de la Pe\~{n}a, Victor H. and Gin\'e, E. Decoupling.
From dependence to independence.
Randomly stopped processes. $U$-statistics and processes. Martingales and beyond. Probability and its Applications (New York). Springer-Verlag, New York, 1999. 

\bibitem{Hiriart:CAMA93} Hiriart-Urruty, J.-B. and Lemar\'echal, C. Convex Analysis and Minimization Algorithms II. 
Advanced theory and bundle methods. Grundlehren der Mathematischen Wissenschaften 306. Springer Verlag.  


\bibitem{Kerkyacharian:Conf09} Kerkyacharian, G.; Mougeot, M.; Picard, D.; Tribouley, K. Learning out of leaders. Multiscale, nonlinear and adaptive approximation, 295--324, Springer, Berlin, 2009.

\bibitem{Koltchinskii:AnnStat09} Koltchinskii, V. Sparse recovery in convex hulls via entropy penalization. Ann. Statist. 37 (2009), no. 3, 1332--1359.

\bibitem{Ledoux:ProbBSp91} Ledoux, M. and Talagrand, M. Probability in Banach spaces.
Isoperimetry and processes. Ergebnisse der Mathematik und ihrer Grenzgebiete (3) [Results in
 Mathematics and Related Areas (3)], 23. Springer-Verlag, Berlin, 1991. xii+480 pp.


\bibitem{Massart:LNM??} Massart, P., Concentration inequalities and model selection. Lectures from 
the 33rd Summer school on Probability Theory in Saint Flour. Lecture Notes in Mathematics, 1896. Springer Verlag (2007).

\bibitem{Oliveira:ArXiv10}
Oliveira, R. I., Concentration of the adjacency matrix and of the laplacian in random graphs 
with independent edges. ArXiv:0911.0600, (2010).

\bibitem{Osborne:IMANA00} Osborne, M. R., Presnell, B. and Turlach, B. A., A new approach to vari-
able selection in least squares problems, IMA Journal of Numerical Analysis
20(3) (2000) 389--404.



\bibitem{DeLaPena:BAMS94}
de la Pe\~{n}a, Victor H., and Montgomery-Smith, S.J. Bounds on the tail probability of $U$-statistics and quadratic forms. Bull. Amer. Math. Soc. (N.S.) 31 (1994), no. 2, 223--227.


\bibitem{Ralph:MOR94} Ralph, D., Global convergence of damped Newton's method for nonsmooth equations via the path search. Math. Oper. Res. 19 (1994), no. 2, 352--389.

\bibitem{Rudelson:IMRN05} Rudelson, M. and Vershynin, R., Geometric approach to error correcting codes
and reconstruction of signals. Int. Math. Res. Not. 64 (2005) 4019--4041.

\bibitem{Rudelson:ICM} Rudelson, M. and Vershynin, R., Non-asymptotic theory of random matrices: extreme singular values. Proceedings of the International Congress of Mathematicians. Volume
 III, 1576--1602, Hindustan Book Agency, New Delhi, 2010. 

\bibitem{Sun:Test10}  Sun, T. and Zhang C.-H., Comments on: $\ell_1$-penalization for mixture regression models, 
Test (2010) 19, 270--275.

\bibitem{Tao:Blog10} Tao, T., The operator norm of a random matrix,
\begin{verbatim}
http://terrytao.wordpress.com/2010/01/09/254a-notes-3-the-operator-norm-of-a-random-matrix/
\end{verbatim}



\bibitem{Tibshirani:JRSSB96} Tibshirani, R. Regression shrinkage and selection via the LASSO, J.R.S.S. 
Ser. B, 58, no. 1 (1996), 267--288. 

\bibitem{Tropp:CRAS08}
Tropp, J. A.  Norms of random submatrices and sparse approximation.
 C. R. Math. Acad. Sci. Paris  346  (2008),  no. 23-24, 1271--1274.

\bibitem{Tropp:ArXiv10}
Tropp, J. A. "User friendly tail bounds for sums of random matrices", http://arxiv.org/abs/1004.4389, (2010).

\bibitem{vdGeer:EJS09} van de Geer, S. and B\"uhlmann, P., On the conditions used to prove oracle results for the Lasso. Electron.
J. Stat. 3 (2009) 1360--1392.

\bibitem{Wainwright:IEEEIT09} Wainwright, Martin J., Sharp thresholds for high-dimensional and noisy sparsity recovery
 using $\ell_1$-constrained quadratic programming (Lasso). IEEE Trans. Inform. Theory  55  (2009), no. 5, 2183--2202.

\bibitem{ZhaoYu:JMLR06} Zhao, P. and Yu, B., On model selection consistency of Lasso.
J. Mach. Learn. Res. 7 (2006), 2541--2563.

\end{thebibliography}

\end{document}